\documentclass[twoside, 11pt]{article}

\usepackage{graphicx, subfigure}
\usepackage[top=1in,bottom=1in,left=1in,right=1in,footnotesep=0.5in]{geometry}
\usepackage[sort&compress, numbers]{natbib} 
\usepackage{amsfonts, amsmath, amssymb, amsthm}
\usepackage{MnSymbol}
\usepackage{enumitem}
\usepackage{comment}

\usepackage{color}
\definecolor{darkred}{RGB}{100,0,0}
\definecolor{darkgreen}{RGB}{0,100,0}
\definecolor{darkblue}{RGB}{0,0,150}

\usepackage{hyperref}
\hypersetup{colorlinks=true, linkcolor=darkred, citecolor=darkgreen, urlcolor=darkblue}
\usepackage{url}

\def\s{{s_*}}
\def\reach{\rho}
\def\level{\cL}
\def\up{\cU}

\def\d{{\rm d}}

\def\ball{B}

\newtheorem{thm}{Theorem}[section]
\newtheorem{theorem}{Theorem}[section]

\newtheorem{lem}[thm]{Lemma}
\newtheorem{lemma}[thm]{Lemma}

\theoremstyle{remark}

\newtheorem{remark}[thm]{Remark}
\newtheorem{algorithm}{Algorithm}
\newtheorem{method}{Method}

\def\beq{\begin{equation}} 
\def\eeq{\end{equation}}
\def\beqn{\begin{eqnarray*}}
\def\eeqn{\end{eqnarray*}}
\def\Bitem{\begin{itemize}\setlength{\itemsep}{.2in}}
\def\bitem{\begin{itemize}\setlength{\itemsep}{.05in}}
\def\eitem{\end{itemize}}
\def\Benum{\begin{enumerate}\setlength{\itemsep}{.2in}}
\def\benum{\begin{enumerate}\setlength{\itemsep}{.05in}}
\def\eenum{\end{enumerate}}
\def\bmult{\begin{multline*}}
\def\emult{\end{multline*}}
\def\bcenter{\begin{center}}
\def\ecenter{\end{center}}
\def\bframe{\begin{frame}}
\def\eframe{\end{frame}}

\newcommand{\thmref}[1]{Theorem~\ref{thm:#1}}

\newcommand{\lemref}[1]{Lemma~\ref{lem:#1}}
\newcommand{\secref}[1]{Section~\ref{sec:#1}}
\newcommand{\figref}[1]{Figure~\ref{fig:#1}}

\newcommand{\algref}[1]{Algorithm~\ref{alg:#1}}
\newcommand{\methodref}[1]{Method~\ref{method:#1}}

\DeclareMathOperator*{\argmax}{arg\, max}

\DeclareMathOperator{\dist}{dist}

\def\cA{\mathcal{A}}
\def\cB{\mathcal{B}}
\def\cC{\mathcal{C}}

\def\cG{\mathcal{G}}
\def\cH{\mathcal{H}}

\def\cL{\mathcal{L}}
\def\cM{\mathcal{M}}

\def\cQ{\mathcal{Q}}

\def\cT{\mathcal{T}}
\def\cU{\mathcal{U}}
\def\cV{\mathcal{V}}
\def\cW{\mathcal{W}}

\def\cZ{\mathcal{Z}}

\def\bbR{\mathbb{R}}

\def\eps{\varepsilon}

\def\1{\mathbbm{1}}

\definecolor{purple}{rgb}{0.4,.1,.9}

\begin{document}
\thispagestyle{empty}

\title{An Asymptotic Equivalence between \\ the Mean-Shift Algorithm~and the Cluster Tree}
\author{
Ery Arias-Castro\footnote{University of California, San Diego, California, USA (\url{https://math.ucsd.edu/\~eariasca/})} 
\and 
Wanli Qiao\footnote{George Mason University, Fairfax, Virginia, USA (\url{https://mason.gmu.edu/\~wqiao/})}
}
\date{}
\maketitle

\begin{abstract}
Two important nonparametric approaches to clustering emerged in the 1970's: clustering by level sets or cluster tree as proposed by Hartigan, and clustering by gradient lines or gradient flow as proposed by Fukunaga and Hosteler.  
In a recent paper \cite{arias2021level}, we argue the thesis that these two approaches are fundamentally the same by showing that the gradient flow provides a way to move along the cluster tree. 
In making a stronger case, we are confronted with the fact the cluster tree does not define a partition of the entire support of the underlying density, while the gradient flow does.
In the present paper, we resolve this conundrum by proposing two ways of obtaining a partition from the cluster tree --- each one of them very natural in its own right --- and showing that both of them reduce to the partition given by the gradient flow under standard assumptions on the sampling density.

\end{abstract}

\section{Introduction} 
\label{sec:introduction}

In the 1970's, two distinctively nonparametric approaches to clustering were proposed, and we argue in this paper that they are fundamentally the same. 

\subsection{The cluster tree}
One of these approaches was the one that Hartigan proposed in his pioneering book on the topic \cite{hartigan1975clustering}, which was published in 1975. In this perspective, the connected components of the upper-level sets of the population density are the candidate clusters, and are seen as ``regions of high density separated from other such regions by regions of low density". 
(Similar takes on clustering emerged around the same time, including in \cite{koontz1972nonparametric}, but historically this perspective has been attributed to Hartigan.)
This view of clustering is not complete in the sense that it depends on a choice of level, which is not at all obvious. Adopting an exploratory stance, Hartigan recommended looking at the entire tree structure, which he called the ``density-contour tree'' and is now better known as the {\em cluster tree}. This structure is the partial ordering between clusters that comes with the set inclusion operation: indeed, as defined, two clusters are either disjoint or one of them  contains the other. 
It is important to note that {\em the cluster tree does not provide a complete partitioning of the space}. It is in fact common to see the region outside the identified clusters as noise, even though in the typical situation this distinction or labeling does not appear natural.
See Figures~\ref{fig:level_sets_1d} and~\ref{fig:level_sets_2d} for illustrations in dimension~1 and~2, respectively.

\begin{figure}[!htpb]
\centering
\includegraphics[scale=0.7]{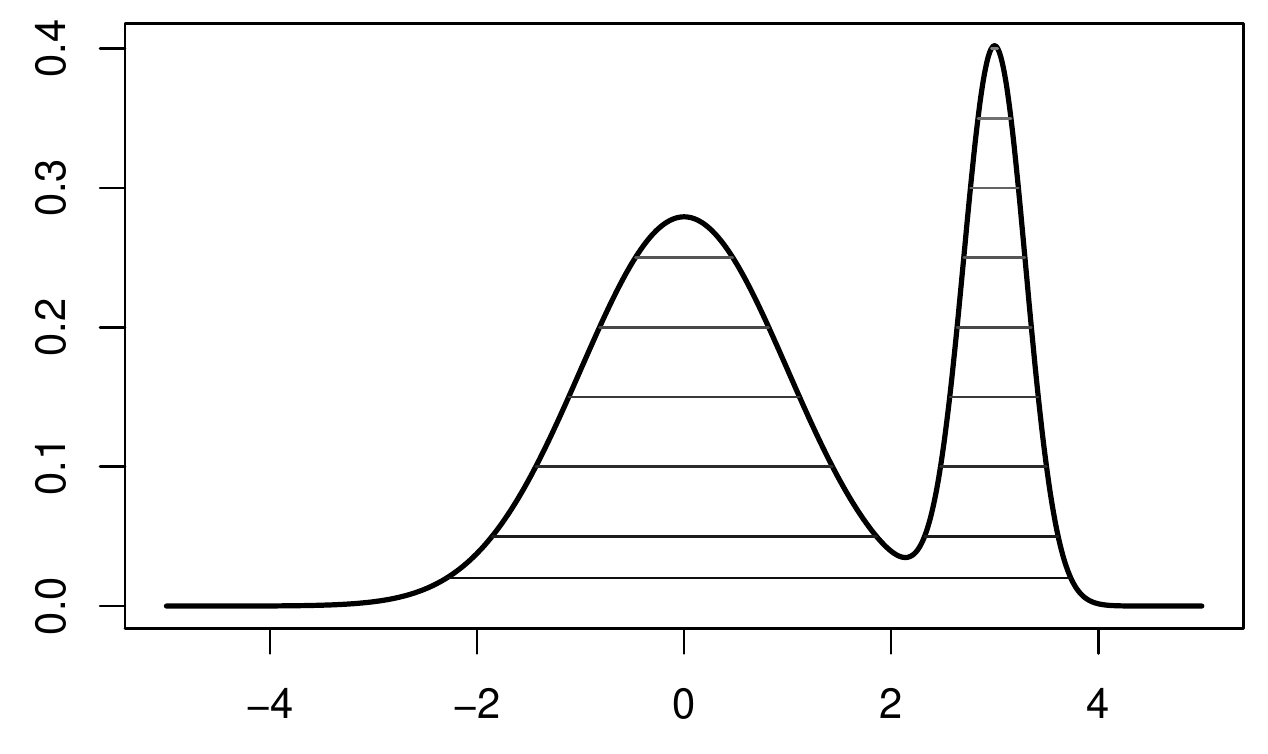}
\caption{A sample of upper level sets of a density in dimension $d = 1$ with two modes (which happens to be the mixture of two normal distributions). At any level $0 < t \le t_0$, where $t_0 \approx 0.0348$ is the value of the density at the local minimum, $\up_t$ is connected, and thus corresponds to the cluster at the level. At any level $t_0 < t \le t_1$, where $t_1 \approx 0.2792$ is the value of the density at its local maximum near $x=0$, $\up_t$ has exactly two connected components, and these are the clusters at that level. (At $t = t_1$, one of the clusters is a singleton.) Finally, at $t_1 < t \le t_2$, where $t_2 \approx 0.4021$ is the value of the density at its global maximum (near $x=3$), $\up_t$ is again connected, and is thus the cluster at that level. 
(At $t = t_2$, the cluster is a singleton.)}
\label{fig:level_sets_1d}
\end{figure}

\begin{figure}[!ht]
\centering
\includegraphics[scale=0.8]{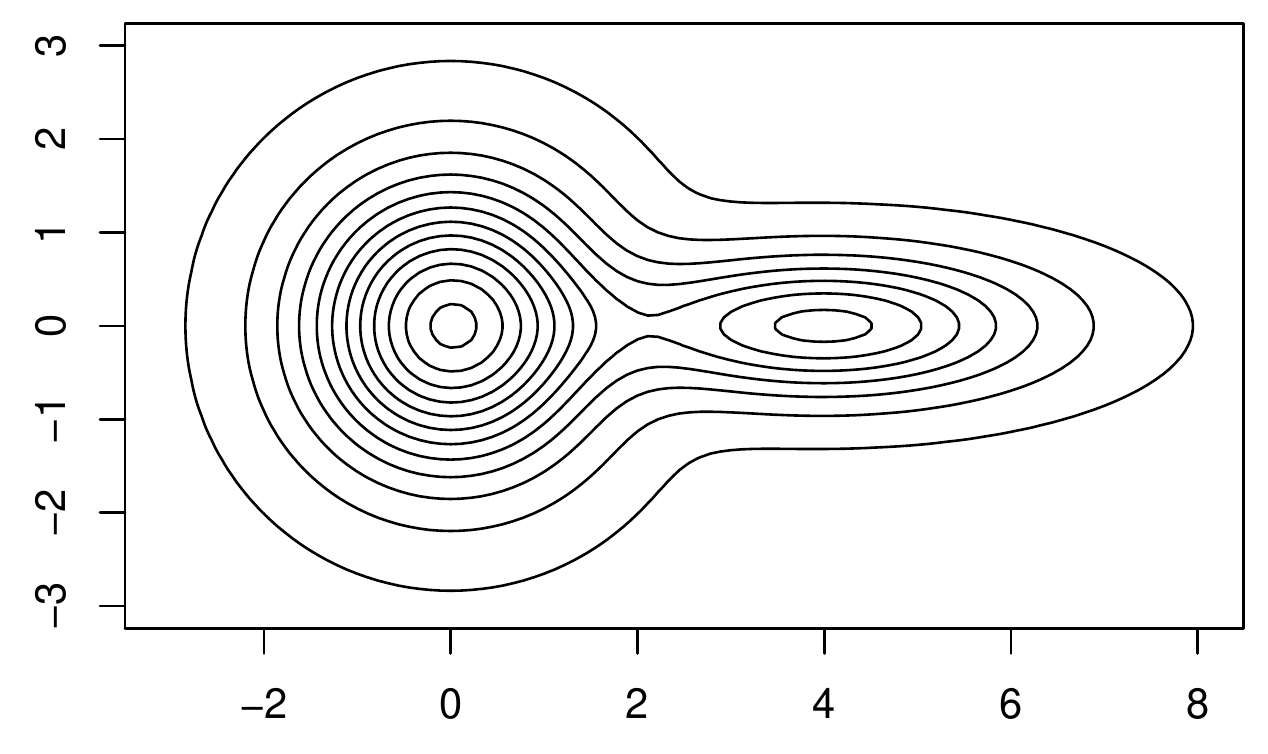}
\caption{A sample of upper level sets of a density in dimension $d = 2$ with two modes (which happens to be the mixture of two normal distributions, one with a non-scalar covariance matrix). The situation is similar to that of \figref{level_sets_1d}, where the number of connected components of the upper level set at $t$ is $=1$ one when $0 < t \le t_0$, where $t_0$ is the value of the density at the saddle point; $=2$ when $t_0 < t \le t_1$, where $t_1$ is the value of the density at the local (but not global) maximum; and $=1$ again when $t_1 < t \le t_2$, where $t_2$ is the maximum value of the density.} 
\label{fig:level_sets_2d}
\end{figure}

Despite not providing an actual clustering of the population (or sample, in practice), the perspective offered in the form of the cluster tree has been influential in statistical research: the estimation of the cluster tree is studied, for example, in \citep{stuetzle2003estimating, stuetzle2010generalized, rinaldo2012stability, chaudhuri2014consistent, wang2019dbscan, eldridge2015beyond}, and the choice of threshold the level is considered, for example, in \citep{sriperumbudur2012consistency, steinwart2015fully, steinwart2011adaptive}.
The estimation of level sets has itself received a lot of attention in the literature \citep{polonik1995measuring, tsybakov1997nonparametric,rigollet2009optimal,chen2017density,mason2009asymptotic,walther1997granulometric,rinaldo2010generalized,singh2009adaptive}.

\subsection{The gradient flow}
\label{sec:intro gradient flow}
The other approach is the one that Fukunaga and Hosteler proposed in an article \cite{fukunaga1975}, incidentally also published in 1975.
In that article, they proposed to use the gradient lines of the population density to move an arbitrary a point upward along the curve of steepest ascent in the topography given by the density. Under some regularity conditions related to Morse theory \citep{chacon2015population}, almost all points in the support of the density (i.e., the population) are in this way moved to a mode (i.e., a local maximum) of the density, and a partition of the population is obtained by grouping together the points ending at the same mode. 
See \figref{gradient_lines_2d} for an illustration.

\begin{figure}[!ht]
\centering
\includegraphics[scale=0.8]{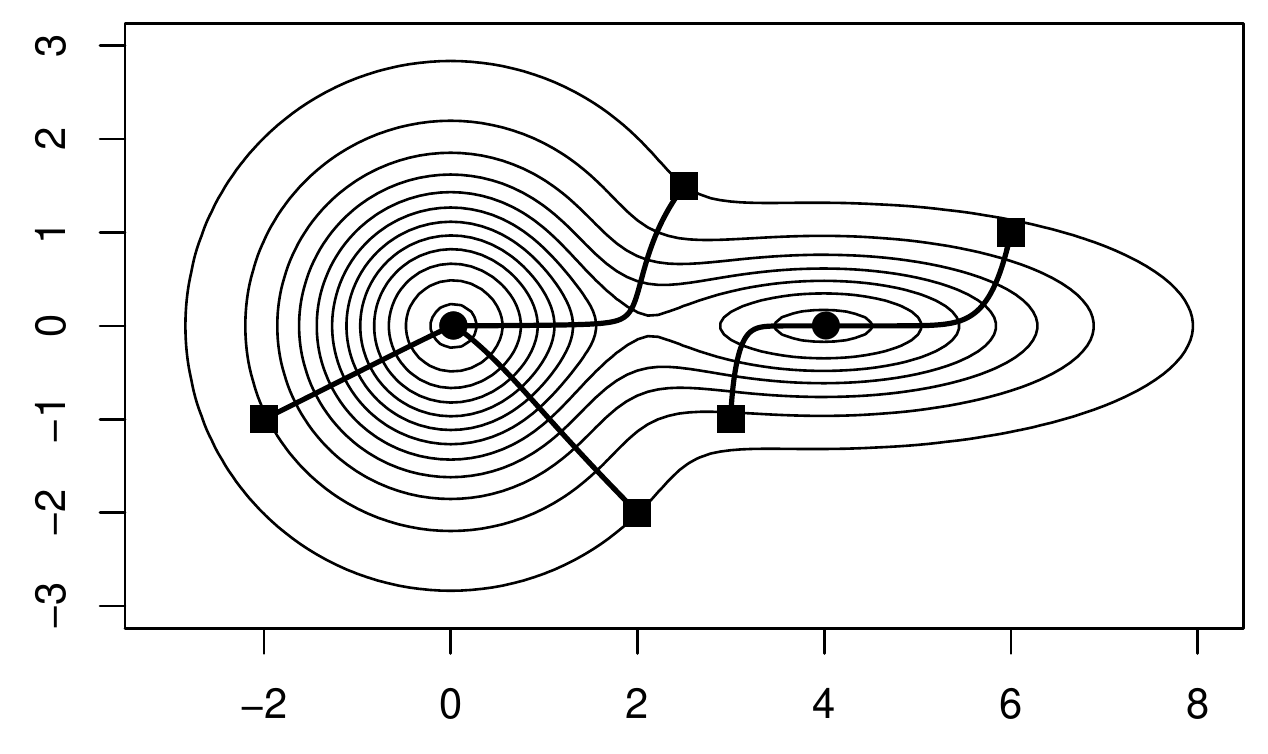}
\caption{A sample of gradient ascent lines for the density of \figref{level_sets_2d}. The square points are the starting points, while the round points are the end points, which are not only critical points but also local maxima (modes) in this example.} 
\label{fig:gradient_lines_2d}
\end{figure}

This approach to clustering has been rediscovered over the years \citep{cheng2004estimating, li2007nonparametric, roberts1997parametric}.
It has also generated a lot of activity on the methodology side. This was started by Fukunaga and Hosteler themselves as they proposed a method in their original paper which is today known as the {\em blurring MeanShift} algorithm.
Interestingly, this method is not strictly plug-in, and in fact behaves quite differently.  Plug-in methods came later, when \citet{cheng1995} proposed what is now known as the {\em MeanShift} algorithm.  
Since then, other methods been proposed, including {\em MedoidShift} \cite{sheikh2007mode}, {\em QuickShift} \cite{vedaldi2008quick}, {\em MedianShift} \cite{shapira2009mode}, and more \cite{carreira2008generalised, carreira2006fast, chacon2019mixture}. The maximum slope approach of \citet{koontz1976graph} is also in that family.
\citet{carreira2015clustering} provides a fairly recent review of this work.
The MeanShift algorithm, and the twin problem of estimating the gradient lines of a density, are now well-understood \citep{cheng1995, comaniciu2002mean, cheng2004estimating, arias2016estimation, carreira2000mode}. 
The behavior of the blurring MeanShift algorithm is not as well understood, although some results do exist \citep{cheng1995, carreira2008generalised, carreira2006fast}.
In this gradient flow approach to clustering, the density modes represent clusters, and this may have also contributed to motivate research work on the estimation of density modes \citep{hartigan1985dip, silverman1981using, minnotte1997nonparametric, dumbgen2008multiscale, burman2009multivariate, genovese2016non}, some of the methods being based on the gradient flow \cite{tsybakov1990recursive, NIPS2017_f457c545}. (Note that modes have been recognized as important features of a density for much longer, at least since \citet{pearson1895}.)

\subsection{Contribution}
The cluster tree proposed by \citet{hartigan1975clustering} and the gradient flow proposed by \citet{fukunaga1975} are clearly related by the central role that modes play in both of them: in the cluster tree, the modes correspond to the leaves; in the gradient flow, the modes correspond to the end points. 
And, indeed, recent survey papers in the area by \citet{menardi2016review, chacon2020modal} discuss them together under the umbrella name of {\em modal clustering}.

In a recent paper \cite{arias2021level}, we go a step further and highlight even stronger ties. We do so by noting that the gradient flow is in a sense compatible with the cluster tree, in that a gradient line does not cross from one cluster to another at the same level and thus respects the partial ordering given by the cluster tree; and by showing how the gradient flow provides a way to move up and down the cluster tree, in the sense that between critical levels, the gradient flow yields a homeomorphism between a cluster and any of its descendants.  

In the present paper, we draw more concrete, and arguably even more closer, ties, which we believe establish these two approaches as being essentially equivalent. By necessity, we work under standard regularity assumptions on the density that guarantee that the gradient flow is well-defined and behaves appropriately. And we do so by suggesting a couple of ways --- which we believe to be very natural --- to obtain a partition from the cluster tree, and then prove that this partition coincides with the partition given by the gradient flow.

The remainder of the paper is organized as follows.
\secref{setting} provides the framework, introduces some concepts and the related notation.
\secref{cluster tree} is the heart of the paper and is where we describe the two algorithms for obtain a partition from the cluster tree (\algref{1} and \algref{2}), and where we state our main results (\thmref{alg1} and \thmref{alg2}).
\secref{discussion} is a discussion section.
The technical proofs are gathered in \secref{proofs}.

\section{Setting and preliminaries}
\label{sec:setting}
Since we discuss approaches to clustering that are both defined based on the sampling density, we discuss and compare them in that context: we have full access to the underlying density, denoted $f$ throughout. All the features that we mention in the paper, in particular (upper) level sets and the cluster tree, gradient lines and the gradient flow, and modes, are always understood in reference to that density.

The density $f$ is with respect to the Lebesgue measure on $\bbR^d$ and is assumed to satisfy the following conditions:
\begin{itemize}
\item {\em Zero at infinity.}
$f$ converges to zero at infinity, meaning, $f(x) \to 0$ as $\|x\| \to \infty$.
\item {\em Twice differentiable.} 
$f$ is twice continuously differentiable everywhere with bounded zeroth, first, and second derivatives.
\item 
{\em Non-degenerate critical points.} 
The Hessian is non-singular at every critical point of $f$.
\end{itemize}
The first condition is equivalent to $f$ having bounded positive (upper) level sets --- see below.
A function that satisfies the last condition is sometimes referred to as a {\em Morse function} \citep{milnor1963morse}.
Similar conditions are standard in the literature cited in the \hyperref[sec:introduction]{Introduction}. 
We discuss possible extensions in \secref{extensions}.

\begin{remark}
The reader will notice that the dimension of the observation space does not play an important role. Of course, in practice, the related methodology has to contend with a standard curse of dimensionality.
\end{remark}

\subsection{Level sets and cluster tree}
For a positive real number $t > 0$, the $t$-level set of $f$ is given by
\begin{equation}
\label{level}
\level_t := \{x : f(x) = t\}.
\end{equation}
while the $t$-upper level set of $f$ is given by
\begin{equation}
\label{up}
\up_t := \{x : f(x) \ge t\}.
\end{equation}
Throughout, whether specified or not, we will only consider levels that are in $(0, \max f)$. Note that, because $f$ converges to zero at infinity and is continuous, its (upper) level sets are compact.

In the level set view of clustering, any connected component of any upper level set is considered a (potential) cluster. The cluster tree, as we consider it here, is then the partial ordering that results from the fact that two clusters are either disjoint or one of them includes the other. (The reader is invited to verify this statement, which comes from the fact that $\up_t \subset \up_s$ when $s \le t$, as is obvious from the definition in \eqref{up}.)

We say that $\cC$ is a {\em cluster} if it is a connected component of an upper level set. We say it is a {\em leaf cluster} if the cluster tree does not branch out past $\cC$, or said differently, if all the descendants of $\cC$ have at most one child.
The last descendant of a leaf cluster is a singleton defined by a mode, and we will use that mode to represent the corresponding lineage.
For a point $x$ and $t \le f(x)$, let $\cC_t(x)$ denote the connected component of $\up_t$ that contains $x$.

\subsection{Gradient lines and gradient flow}
\citet{fukunaga1975} suggest to ``assign each [point] to the nearest mode along the direction of the gradient". Formally, the gradient ascent line starting at a point $x$ is the curve given by the image of $\gamma_x$, the parameterized curve defined by the the following ordinary differential equation (ODE)
\begin{equation} \label{gradient_flow}
\gamma_x(0) = x; \quad
\dot\gamma_x(t) = \nabla f(\gamma_x(t)).
\end{equation} 
Under our conditions on $f$, $\nabla f$ is Lipschitz, and this is enough for standard theory for ODEs \citep[Sec 17.1]{hirsch2012differential} to justify this definition. This, and the fact that this is a gradient flow \citep[Sec 9.3]{hirsch2012differential}, gives the following.

\begin{lemma}
\label{lem:gradient flow}
For any $x$, the function $\gamma_x$ defined in \eqref{gradient_flow} is well-defined on $[0,\infty)$, with $\gamma_x(t)$ converging to a critical point of $f$ as $t \to \infty$.
\end{lemma}

Checking that $s \mapsto f(\gamma_x(s))$ is non-decreasing is simply due to 
\begin{equation}
\label{density_increase}
\frac{\d }{\d s} f(\gamma_x(s)) = \|\nabla f(\gamma_x(s))\|^2 \ge 0.
\end{equation}
The {\em basin of attraction} of a point $x_0$ is defined as $\{x: \gamma_x(\infty) = x_0\}$. Note that this set is empty unless $x_0$ is a critical point (i.e., $\nabla f(x_0) = 0$).

In the gradient line view of clustering, we call a cluster any basin of attraction of a mode.
It turns out that, if $f$ is a Morse function \citep{milnor1963morse}, then all these basins of attraction, sometimes called stable manifolds, provide a partition of the support up to a set of zero measure.
\begin{lemma}
\label{lem:basins}
Under the assumed regularity conditions, the basins of attraction of the local maxima, by themselves, cover the population, except for a set of zero measure.
\end{lemma}
Indeed, by \lemref{gradient flow}, the basins of attraction partition the entire population. In addition, the set of critical points is discrete \citep[Cor~3.3]{banyaga2013lectures}, the basin of attraction of each critical point that is not a local maximum is a (differentiable) submanifold of co-dimension at least one\footnote{ The requirement in that theorem that the function be compactly supported is clearly not essential.} \cite[Th~4.2]{banyaga2013lectures}, and therefore has zero Lebesgue measure.
For more background on Morse functions and their use in statistics, see the recent articles of \citet{chacon2015population} and \citet{chen2017statistical}.

In their paper, \citet{fukunaga1975} propose to discretize the following variant of the gradient ascent flow 
\begin{equation} \label{gradient_flow_ms}
\xi_x(0) = x; \quad
\dot\xi_x(t) = \frac{\nabla f(\xi_x(t))}{f(\xi_x(t))}.
\end{equation}
Note that this is the gradient flow given by $\log f$, while the flow defined in \eqref{gradient_flow} is the gradient flow given by $f$ itself. The rationale given in \cite{fukunaga1975} is that the latter helps move points in low-density regions upward more quickly, but note that, regardless, the gradient lines defined by these two flows coincide --- only the speed at which a line is traveled changes. 
Later in the paper, we use another variant: the gradient line parameterized by the level, which takes the form
\begin{equation} \label{gradient_flow_norm}
\zeta_x(0) = x; \quad
\dot\zeta_x(t) = 
\frac{\nabla f(\zeta_x(t))}{\|\nabla f(\zeta_x(t))\|^2}.
\end{equation}
Indeed, $\dot\zeta_x(t) \propto \nabla f(\zeta_x(t))$, so that $\zeta_x$ traces the same gradient line as $\gamma_x$ defined in \eqref{gradient_flow}, and elementary derivations show that $f(\zeta_x(t)) = t- f(x)$ for all applicable $t$, knowing that here the flow is only defined for $t \in [0, f(x_*)-f(x))$, where $x_*$ denotes the critical point to which $x$ is attracted. 
In \lemref{converge mode}, we show that $\zeta_x(t)$ converges to $x_*$ as $t \nearrow f(x_*)-f(x)$, as expected.

\begin{remark}
The difference between the MeanShift algorithm and the blurring MeanShift algorithm does not arise at the population level, but at the sample level, and has to do, roughly, with whether the data points remain fixed or are updated with each iteration of a Euler discretization of \eqref{gradient_flow_ms}.
\end{remark}

\subsection{Metric projection and reach}

The concepts of metric projection and reach, the latter introduced by \citet{federer1959curvature}, will play an important role. For a point $x$ and $\delta > 0$, $\ball(x,\delta)$ is the open ball centered at $x$ of radius $\delta$.
For $\cA \subset \bbR^d$ and $\delta > 0$, define its $\delta$-neighborhood as 
\begin{equation}
\ball(\cA, \delta) 
:= \bigcup_{a \in \cA} \ball(a, \delta)
= \{x : \dist(x, \cA) < \delta\}, 
\end{equation}
where $\dist(x, \cA) := \inf\{\|x-a\| : a \in \cA\}$. 
The projection (aka metric projection) of a point $x$ onto a closed set $\cA$ is the subset points in $\cA$ that are closest to $x$. 
That subset is nonempty if $\cA$ is non-empty.
We say that $x$ has a unique projection onto $\cA$ if its projection is a singleton. 
The reach of $\cA \subset \bbR^d$, denoted $\reach(\cA)$, is defined as the infimum over $\delta>0$ such that every point in $\ball(\cA, \delta)$ has a unique projection onto $\cA$. 
Thus the projection on a set $\cA$ is well-defined on $\ball(\cA, \reach(\cA))$.

\section{Cluster tree: A successive projection algorithm}
\label{sec:cluster tree}

The cluster tree was presented by \citet{hartigan1975clustering} as a partial order structure on the clusters defined by the level sets. To this day, there is no clear proposal for a partition of the entire population based on the cluster tree. 
(When we say `entire population', we mean the support of the density understood as being up to a set of measure zero.)
We present below what we see as a natural definition. The core idea is intuitive enough: each point in the population is moved upward the cluster tree to a leaf mode. The implementation of this idea is not as straightforward, although it remains natural: discretize the levels and project on successively higher and higher level sets. It turns out that, as the level discretization becomes finer and finer, this algorithm is well-defined for a larger and larger proportion of the population, covering almost all the population, eventually. 
And this behavior does not depend on the particular discretization.
Moreover, the limiting sequence obtained in this fashion converges in a proper sense to the gradient ascent line with origin the starting point, implying, therefore, that clustering according to the cluster tree --- if understood as doing the above --- coincides with clustering according to the gradient flow.
And this is our central message, announced in the \hyperref[sec:introduction]{Introduction}.

\subsection{Moving up the cluster tree: Algorithm~1}
\label{sec:alg1}

We define a way --- which we believe to be very natural --- of moving up the cluster tree, starting at an arbitrary point in the population and ending at a leaf cluster. This process will result in a partitioning of the entire population. 
In this first version, we discretize the levels. Let $\eta$ denote the size of the discretization step. 

\begin{algorithm}
\label{alg:1}
Given any point $x$ in the support of $f$, define the sequence of levels: set $t_0 = f(x)$, and for $k \ge 1$, let $t_k = t_0 + k \eta$. 
Then build the following sequence of points: set $q_0 = x$, and for $k \ge 1$, let $q_k$ denote a metric projection of $q_{k-1}$ onto $\level_{t_k}$; if $q_k \notin \cC_{t_{k-1}}(q_{k-1})$ or if $\level_{t_k} = \emptyset$, stop and return $\argmax_{\,\cC_{t_{k-1}}(q_{k-1})} f$. 
\end{algorithm}

Note that the algorithm must end in a finite number of steps since the level increases by $\eta$ are every step.
The condition that the new projection point remains inside the cluster associated with the current point is there to ensure that the sequence respects the partial ordering that is the cluster tree, and in particular prevents long jumps across the cluster tree.

We also remark that other options include stopping when a projection is not unique; stopping when reaching a leaf cluster; and pursuing all possible paths when a projection is not unique and returning all the leaf clusters reached in that a way. It turns out that these options are essentially equivalent for almost all the points in the population under the regularity assumed of the density.

\begin{theorem}
\label{thm:alg1}
Suppose that $x$ is in the basin of attraction of a mode $x_*$. 
Then, for $\eta$ small enough, \algref{1} given $x$ as starting point returns the mode $x_*$ 
Furthermore, as $\eta \to 0$, the polygonal line connecting the sequence $(q_k)$ computed by the algorithm converges to the gradient ascent line originating at $x$.
\end{theorem}

The first part of the theorem is established by showing that the sequence that the algorithm computes enters a leaf cluster containing $x_*$, and once this is established, the sequence must remain there by construction. This is done by comparing the sequence with the following sequence sampled from the gradient line: $z_0 = x$, and for $k \ge 1$, $z_k = \zeta_x(t_k - t_0)$.
In the second part, convergence is with respect to the Hausdorff metric for subsets in $\bbR^d$ --- see \eqref{d_H}.

By \lemref{basins}, the result applies to almost all points in the support of $f$. The theorem thus shows that the way of partitioning the population according to the cluster tree provided by \algref{1} in the infinitesimal limit $\eta\to0$ --- which again we find rather natural --- is equivalent to partitioning the population according to the gradient flow.

\subsection{Moving up the cluster tree: Algorithm~2}
\label{sec:alg2}

We look at another natural way of moving up the cluster tree. In this second variant, we control the spatial step size as opposed to the level step size in the previous algorithm. 
Let the spatial step size be $\eps > 0$.

\begin{algorithm}
\label{alg:2}
Given a point $x$ in the population, successively project onto the highest level set within distance $\eps$: set $t_0 = f(x)$ and $q_0 = x$, and for $k \ge 1$, let $t_k = \max\{f(y) : \|y - q_{k-1}\| \le \eps\}$, and let $q_k$ denote a projection of $q_{k-1}$ onto $\level_{t_k}$; if $q_k = q_{k-1}$, stop and return that point.
\end{algorithm}

Unlike the previous algorithm, it is not obvious that \algref{2} will end in a finite number of steps.
We note that, as before, other options are possible, but again, they would be equivalent for almost all points in the population. 

\begin{theorem}
\label{thm:alg2}
Suppose that $x$ is in the basin of attraction of a mode $x_*$. 
Then, for $\eps$ small enough, \algref{2} given $x$ as starting point returns the mode $x_*$. 
Furthermore, as $\eps \to 0$, the polygonal line connecting the sequence $(q_k)$ computed by the algorithm converges to the gradient ascent line originating at $x$.
\end{theorem}

The first part of the theorem is established by showing that the sequence that the algorithm computes converges to the mode. 
We note that such a convergence, if it happens, must happen in a finite number of steps since the moment $x_*$ is within distance $\eps$ of $q_k$, we have $q_{k+1} = x_*$ and the process stop.
The technical arguments are otherwise similar to those underlying \thmref{alg1}.

As before, by \lemref{basins}, the result above applies to almost all points in the population, thus establishing that partition given by \algref{2} in the infinitesimal limit $\eta\to0$ --- another natural way of dividing the support according to the cluster tree --- is equivalent to partitioning the population according to the gradient flow.

\section{Discussion}
\label{sec:discussion}
\subsection{Comparison with Euler's methods}
The forward and the backward Euler methods are two well-known numerical methods for solving ODEs. \citet{fukunaga1975} suggested to use the forward Euler method  to approximate the gradient flow of \eqref{gradient_flow_ms}: for a fixed small $\eps>0$, starting from a given point $x_0=x$, successively compute
\begin{equation}
x_{i+1} = x_i + \eps \frac{\nabla f(x_i)}{f(x_i)},\quad i=0,1,2, \dots,
\end{equation}
until convergence, where the gradient of $\log f$ at the current location is explicitly used to generate the next point in the sequence. 
As remarked in \secref{intro gradient flow}, the method that they proposed --- now known as the blurring MeanShift algorithm --- is not a plug in implementation of this. Such a method came later from \citet{cheng1995} --- a method now known as the MeanShift algorithm.
A more detailed discussion of the connection between the MeanShift algorithm and the forward Euler method is given in \citep{arias2016estimation}.

In contrast to the forward Euler method, the backward Euler method is an implicit approach: for a fixed small $\eps>0$, starting from a given point $x_0=x$, the backward Euler method computes
\begin{equation}
\label{backward_Euler}
x_{i+1} = x_i + \eps \nabla f(x_{i+1}),\quad i=0,1, 2, \dots,
\end{equation}
until convergence. Below we argue that our \algref{1} and \algref{2} can be viewed as two variants of the backward Euler method. The backward Euler method is known to be stable and the generated sequence has a monotonic property, which is also shared by our two algorithms.

\algref{1}: recall that before the algorithm stops, $q_k$ is a metric projection of $q_{k-1}$ onto $\cL_{t_k}$, where $t_k = t_0 + k\eta$. When $\eta$ is small enough, $q_{k-1} - q_k$ is normal to $\cL_{t_k}$ at $q_k$, so that we can write
\begin{equation}
\label{Alg1_equation}
q_k = q_{k-1} + \eps_k\nabla f(q_k),
\end{equation}
where, this time, $\eps_k$ depends on $k$ and is determined by the property that it is smallest such as to guarantee that $q_k\in \cL_{t_k}$. The dynamical systems \eqref{backward_Euler} and \eqref{Alg1_equation} are otherwise the same. It is in this sense that we view \algref{1} as being a variant of the backward Euler method.

\algref{2}: similar to \algref{1}, $q_k$ is a projection of $q_{k-1}$ onto $\level_{t_k}$ but using a different definition of $t_k = \max\{f(y) : \|y - q_{k-1}\| \le \eps\}$. We can, of course, use the expression in \eqref{Alg1_equation} again to interpret \algref{2} as a variant of the backward Euler method. There is also a different perspective: In \algref{2}, the metric projection of $q_{k-1}$ to $\level_{t_k}$ and the maximization of $f$ over $\partial\ball(q_{k-1}, \eps)$ are in fact the same thing except for the very last step. (This is not completely obvious, but not hard to anticipate, and we omit further details.) That is, the projection point $q_k$, unless it is the final output, is the solution to the following maximization problem with a constraint: 
\begin{align}
\label{Lagrange}
\text{maximize } f(y), \text{ subject to } \|y - q_{k-1}\|^2 = \eps^2.
\end{align}
With a Lagrange multiplier $\lambda$, the corresponding Lagrange function of the above problem is 
\begin{equation}
L(y,\lambda) := f(y) - \lambda (\|y - q_{k-1}\|^2 - \eps^2).
\end{equation}
Taking the gradient of $L$ with respect to $y$ and equating it to 0, we obtain
\begin{equation}
\label{Lagrange_gradient}
\nabla f(y) = 2\lambda (y-q_{k-1}).
\end{equation}
Since we know $q_k$ is the solution, we can write
\begin{equation}
\label{Alg2_equation}
q_k = q_{k-1} + \frac{1}{2\lambda} \nabla f(q_k ),
\end{equation}
where $\lambda$ has to be solved using \eqref{Lagrange_gradient} and the constraint in \eqref{Lagrange}. 
We have thus argued that, compared with the original backward Euler method in \eqref{backward_Euler}, \eqref{Alg2_equation} can also be considered as a variant with a varying step size, which here needs to be determined to guarantee that $\|q_k - q_{k-1}\| = \eps$ before the very last step.

\subsection{Milder assumptions on the density}
\label{sec:extensions}
The assumptions that we made of the density are very standard in this literature, and they are mostly driven by the necessity to have gradient lines be well-defined and well-behaved, and to have Morse theory apply. 
But if this can be guaranteed in other ways, then it seems possible to extend the results to other densities.

A case in point --- one that is relatively easy to address --- is the assumption that the modes are non-degenerate. It's not hard to extend the first part of \thmref{alg1} and first part of \thmref{alg2} --- which are the main parts since they imply the equivalence that is the thesis of this paper --- to a setting where some modes are degenerate. We do assume that the modes are isolated, however, and this is now a separate assumption. (Non-degenerate modes are automatically isolated.) 
Indeed, take \algref{1}, and notice that the first part of \thmref{alg1} is stated for fixed $\eta$. Then the idea is very simple: modify the density within the very last level set containing the mode (or the second-to-last if the last level corresponds to the mode's level) in such a way that the resulting function has a non-degenerate mode at the same location and no other mode within the region of modification. Then, since the algorithm builds the sequence without having to access the values of the density inside that region, the sequence it builds is identical up to the second-to-last step, but also the very last step since it also corresponds to jumping directly to the mode.
Formalizing this is rather straightforward, but tedious, and we do not provide further details.

Another similar extension is to situations in which the density has a flat top. This is a situation in which a level set has non-empty interior, in which case each connected component of such a level set that has non-empty interior is considered a flat top and is for an end leaf in the cluster tree. The algorithms that we have considered continue to work.
From the gradient flow perspective, points whose gradient ascent lines end at the same flat top are clustered together. 
And it is not hard to see that the equivalence remains the same in that the first part of \thmref{alg1} and first part of \thmref{alg2} continue to apply.
This can be argued in the same way, by modifying the density a little bit within each flat top to make it of Morse-type.
Again, making this rigorous is not hard to do, but laborious.

Where the results do not apply, and in fact the equivalence because rather questionable, is when the density has discontinuities. In that case, clustering by gradient flow seems less appropriate (even if the density is piecewise smooth), as the discontinuity imply a saliency of a larger magnitude that is better captured by a level set approach to clustering.

\subsection{Bound on the Hausdorff distance}
\label{sec:Hausdorff bound}
We designed \algref{1} and \algref{2} as natural ways to obtain a partition from the cluster tree, and the analysis we provide for them in \thmref{alg1} and \thmref{alg2} establishes that both are equivalent to the partition defined by the gradient flow when the step size converges to zero --- which really is also natural given the context. As this was our main goal, this qualitative analysis is enough for our purposes.

However, it turns out that our proof arguments can be easily augmented to yield a convergence rate for the polygonal line generated by any one of these two algorithms and the corresponding gradient line. A relatively small elaboration of these arguments yields that {\em the Hausdorff distance between these two lines is $O(\sqrt{\eta})$ for \algref{1} and $O(\eps)$ for \algref{2}.} 

We provide more formal arguments establishing these rates in \secref{proof1 Hausdorff bound} and \secref{proof2 Hausdorff bound}.
We do not know if these rates are sharp, but this is very tangential to our main purpose here.

\subsection{Uniform consistency of \algref{1} and \algref{2}}
\label{sec:uniform}
\thmref{alg1} and \thmref{alg2} only give pointwise (i.e., for a fixed $x$ as the starting point) convergence results for the two algorithms. 
It turns out that this convergence behavior is also uniform in the following sense. {\em For \algref{1}, there is a measurable set $\Omega_\eta$ with probability at least $1-g(\eta)$ for some $g(\eta) \to 0$ as $\eta \to 0$ such that \algref{1} applied to any $x \in \Omega_\eta$ returns the associated mode, meaning, $\lim_{t\to\infty}\gamma_x(t)$. And an analogous result applies for \algref{2}.}

We give more formal arguments in \secref{proof1 uniform} and \secref{proof2 uniform}. 

A rate of convergence of the order of $g(\eta)=O(\sqrt{\eta})$ can also be obtained if one adopts the framework of \citet{chen2017statistical}, and assumes the same conditions of the density, including the assumption (D) there that requires that ``the gradient flow to move away from the boundaries of clusters", where a ``cluster" refers to the basin of attraction of a mode. Other conditions include that $f$ is a Morse function with bounded support and is thrice differentiable with bounded third derivatives. Details are omitted.

\subsection{Methodology inspired by \algref{1} and \algref{2}}
\label{sec:methods}
\algref{1} and \algref{2} were proposed as ways to obtain a partition of the population which, based on successive projections onto higher and higher level sets, may merit to be considered `natural'. And we showed that both ways lead to the partition given by the gradient flow. This equivalence between level set or cluster tree clustering and gradient lines or gradient flow is exactly what we were aiming for.

At the same time, these algorithms can be easily turned into methods by simply applying them on an estimate of the density instead of the density itself when, in the usual framework where only a sample is available, the latter is not available. This plug-in approach yields the following two clustering methods.

\begin{method}
\label{method:1}
Given a sample, $x_1, \dots, x_n \in \bbR^d$, estimate the density to obtain $\hat f$. 
Given some $\eta>0$, apply \algref{1} to each data point with $\hat f$ in place of $f$: let $c_i$ denote what the algorithm returns when applied to $x_i$.
Group the data points according to the distinct elements in $\{c_1, \dots, c_n\}$.
\end{method}

\methodref{1} can be related to the {\em QuickShift} algorithm proposed by \citet{vedaldi2008quick}. 
QuickShift works as follows: based on a sample and an estimate of the density, starting at an arbitrary point, the algorithm iteratively moves to the closest data point whose estimated density is strictly larger than the current value, as long as this data point is within a given distance. 
\citet{NIPS2017_f457c545} and \citet{jiang2018quickshift++} examined the theoretical properties of this method and a variant.
In \methodref{1}, the points on a trajectory are forced to be on certain level sets. This guarantees that the level increases by $\eta$ with every step except possibly for the very last step. 
Note that, in our context, this was motivated from a need to discretize the process of moving up the cluster tree and not for methodological reasons.

\begin{method}
\label{method:2}
Same as \methodref{1} but with \algref{2} (provided with $\eps>0$) in place of \algref{1}.
\end{method}

\methodref{2} can be related to the {\em MaxShift} algorithm that we are proposing and in the process of studying (work in progress at the moment of writing). 
MaxShift is a kind of dual to QuickShift, and works as follows: based on a sample and an estimate of the density, starting at an arbitrary point, the algorithm iteratively moves to the data point whose estimated density is largest, as long as this data point is within a given distance. 
In \methodref{2}, the process is, at least in appearance, a bit more careful in that it has the point move to the closest point at the highest level achieved within a certain range. 
Note that, again, this design choice is motivated not by methodological considerations, but by the main goal of this paper, which was to propose natural ways of moving up the cluster tree --- and then showing that these are in a sense equivalent to follow the gradient lines upward.

\medskip
In any case, we are not proposing these methods as competitors of these or other MeanShift variants, as their implementation appears to be cumbersome and computationally intensive because of having to estimate very many (as $\eta$ or $\eps$ are chosen small) level sets and having to project onto them.
The estimation of the density can be done, for example, by kernel density estimation with bandwidth chosen by leave-one-out cross-validation. Even after having chosen the bandwidth automatically, there remains the choice of $\eta$ or $\eps$, which is analogous to the choice of step size in a forward Euler implementation of the gradient flow --- the MeanShift method of \citet{cheng1995}. This choice is nontrivial, as is often the case for the choice of tuning parameter(s) in a clustering method.

All that being said, these methods can be shown to be consistent. 
This can be developed in three steps in a typical way as, for example, detailed in the proof of \citep[Th~2]{arias2016estimation}. 
Below we give a brief sketch for these three steps with the intention of convincing the reader that a formal proof is not at all far beyond. 
\begin{itemize}
\item {\em Step 1:} By estimating $f$ by kernel density estimation with an appropriate kernel, for example, we obtain $\hat f$ that satisfies all the same assumptions we have required for $f$. Therefore, \thmref{alg1} and \thmref{alg2} apply to $\hat f$, that is, the polygon lines connecting the sequences generated in \methodref{1} and \methodref{2} are close approximations to the gradient flow lines induced by $\hat f$. 
\item {\em Step 2:} When the sample size is large enough, $\hat f$ and $f$, and their derivatives up to the second order, are close to each other, which further implies by some classical stability result for ODEs \citep[Sec 17.5]{hirsch2012differential}, that their induced gradient flow lines are consistent. 
\item {\em Step 3:} Combining the results in Step 1 and Step 2, we can then immediately say that the polygon lines connecting the sequences in \methodref{1} and \methodref{2} are consistent estimators of the gradient flow line of $f$, if their starting points are the same, and the gradient flow line ends at a mode.
\end{itemize}
In fact, the rates discussed in \secref{Hausdorff bound} in the population setting can be shown --- again following standard arguments --- to transfer to the sample setting.

\section{Technical details}
\label{sec:proofs}

Recall the assumptions placed on the density $f$ in \secref{setting}. In particular, it is differentiable and its gradient is bounded and is Lipschitz. 
We let $\kappa_1$ denote the sup norm of the gradient and let $\kappa_2$ denote its Lipschitz constant. Then $f$ is $\kappa_1$-Lipschitz, meaning
\begin{align}
\label{kappa1}
|f(y) - f(x)| \le \kappa_1 \|y-x\|, \quad \forall x, y,
\end{align}
and $\nabla f$ is $\kappa_2$-Lipschitz, meaning
\begin{align}
\label{kappa2}
\|\nabla f(y) - \nabla f(x)\| \le \kappa_2 \|y-x\|, \quad \forall x, y,
\end{align}
and the following Taylor expansion holds
\begin{align}
\label{taylor2}
\big|f(y) - f(x) - \nabla f(x)^\top (y-x)\big|
\le \tfrac12 \kappa_2 \|y-x\|^2, \quad \forall x, y.
\end{align}
We will denote $N(x) := \nabla f(x)/\|\nabla f(x)\|$, which gives the inward-pointing unit normal of $\level_t = \{f = t\}$ at $x$ whenever $\nabla f(x) \ne 0$.

Throughout, we let 
\begin{align}
\label{F}
F(x) := \frac{\nabla f(x)}{\|\nabla f(x)\|^2},
\end{align}
which is the function that defines the ODE given in \eqref{gradient_flow_norm}.

When we speak of distance between sets we mean the Hausdorff distance, which for $\cA, \cB \subset \bbR^d$ is defined as
\begin{align} \label{d_H}
d_H(\cA, \cB) 
:= \max\big\{d_H(\cA \mid \cB), d_H(\cB \mid \cA)\big\}, &&
d_H(\cA \mid \cB)
:= \sup_{a\in \cA} \inf_{b \in \cB} \|a-b\|.
\end{align}

\subsection{Preliminaries}

%
Following \cite{federer1959curvature}, for a set $\cA$ and $x \in \cA$, $\reach(\cA, x)$ denotes the supremum over $r > 0$ such that the metric projection onto $\cA$ is well-defined in the whole of $\ball(x, r)$. By definition, $\reach(\cA)$ is the infimum of $\reach(\cA, x)$ over $x \in \cA$. 

The following is an immediate consequence of \citep[Lem 4.8(2)]{federer1959curvature}.

\begin{lem} \label{lem:federer4.8(2)}
Take a set $\cA$ and $x \in \cA$ such that $r := \reach(\cA, x) > 0$. Then $\cA$ admits a tangent space at $x$.
Also, the metric projection onto $\cA$, denoted $P$, is well-defined on $\ball(x, r)$, and for any $y \in \ball(x,r)$, $P(y) = x$ and $x-y$ is orthogonal to the tangent space of $\cA$ at $x$. 
\end{lem}
  
The following is a quantitative version of \citep[Lem 4.11]{federer1959curvature}, distilled from its own proof.

\begin{lem} \label{lem:federer4.11}
For any level $t>0$ and any $x \in \level_t$, $\reach(\level_t, x) \ge (1/4\kappa_2) \|\nabla f(x)\|$. 
\end{lem}




\begin{lemma}
\label{lem:dist}
Take $x$ such that $\nabla f(x) \ne 0$ and let $t := f(x)$. Then, if $\eta > 0$ is small enough that $\eta \le \|\nabla f(x)\|^2/2\kappa_2$, it holds that $\dist(x, \level_{t+\eta}) \le 2 \eta/\|\nabla f(x)\|$.
\end{lemma}

\begin{proof}
By \eqref{taylor2}, we have, for $u > 0$, 
\begin{align}
\label{taylor_N}
f(x+ uN(x)) \ge f(x) + u \|\nabla f(x)\| - \tfrac12 \kappa_2 u^2.
\end{align}
Applying this with $u_0 = \|\nabla f(x)\|/\kappa_2$, and letting $x_0 := f(x+ u_0 N(x))$,  gives $f(x_0) \ge t + \|\nabla f(x)\|^2/2\kappa_2$.
Because $f$ is continuous on the half-line $\{x+ \lambda N(x) : \lambda > 0\}$, if $\eta \le \|\nabla f(x)\|^2/2\kappa_2$, then necessarily, there is $u_1 \le u_0$ such that $f(x+u_1 N(x)) = t+\eta$, which then forces $\dist(x, \level_{t+\eta}) \le \|x - (x+u_1 N(x))\| = u_1$. And because $f(x+u_1 N(x)) \ge t + u_1 \|\nabla f(x)\|/2$, we also have that $u_1 \le 2 \eta/\|\nabla f(x)\|$.
\end{proof}

Let $P_t$ denote the metric projection onto $\level_t$, which is well-defined on $\ball(\level_t, \reach(\level_t))$.
Combining the last two lemmas, we get the following
\begin{lemma}
\label{lem:proj}
Take $x$ such that $\nabla f(x) \ne 0$ and let $t := f(x)$. Then, if $\eta > 0$ is small enough that $\eta \le \|\nabla f(x)\|^2/16\kappa_2$, it holds that $P_{t+\eta}(x)$ is well-defined and satisfies $\|P_{t+\eta}(x) - x\| \le 2 \eta/\|\nabla f(x)\|$.
\end{lemma}

\begin{proof}
The conditions of \lemref{dist} are met, so if $y \in \cL_{t+\eta}$ is closest to $x$, then $\|y-x\| \le 2 \eta/\|\nabla f(x)\|$. Now, by \eqref{kappa2},
\begin{align}
\|\nabla f(y)\| 
&\ge \|\nabla f(x)\| - \kappa_2 \|y-x\| \\
&\ge \|\nabla f(x)\| - \kappa_2 (2 \eta/\|\nabla f(x)\|) \\
&> \tfrac12 \|\nabla f(x)\|.
\end{align}
Calling in \lemref{federer4.11}, and using the assumed bound on $\eta$, we derive 
\[\reach(\cL_{t+\eta}, y) 
> (1/4\kappa_2) \|\nabla f(y)\| 
\ge (1/8\kappa_2) \|\nabla f(x)\|
\ge 2 \eta/\|\nabla f(x)\|
\ge \|y - x\|,\]
implying, by definition, that $y$ is the unique projection of $x$ onto $\cL_{t+\eta}$.
\end{proof}

Elementary derivations yield the following.
\begin{lemma}
\label{lem:F}
When $f$ is twice differentiable, the function $F$ defined in \eqref{F} is once differentiable at any $x$ where $\nabla f(x) \ne 0$, where its derivative is equal to 
\begin{align}
D F(x) = \frac{H f(x)}{\|\nabla f(x)\|^2} - \frac{2 \nabla f(x) \nabla f(x)^\top H f(x)}{\|\nabla f(x)\|^4}.
\end{align}
In particular, $\|D F(x)\| \le 3 \|H f(x)\|/\|\nabla f(x)\|^2$ at any such $x$.
\end{lemma}

%

\begin{lem}
\label{lem:P}
There is a constant $C > 0$ that depends on $f$ such that the following holds.
Consider a level $t > 0$ and $x \in \level_t$ such that $\nabla f(x) \ne 0$.
If $\eta > 0$ is small enough that $\eta \le (1/C) \|\nabla f(x)\|^2$, then $P_{t+\eta}(x)$ is well-defined and satisfies
\begin{equation}
\label{P}
P_{t+\eta}(x) =  x + \eta F(x) \pm \eta^2 \frac{C}{\|\nabla f(x)\|^3}.
\end{equation}
\end{lem}


\begin{proof}
Below, $x$ and $t$ are considered fixed. 
Let $C$ be large enough that $\eta \le \|\nabla f(x)\|^2/16\kappa_2$, so that \lemref{proj} applies to give us that $x_\eta := P_{t+\eta}(x)$ is well-defined and satisfies 
\begin{align}
\label{P_proof1}
\|x_\eta - x\| 
\le 2\eta/\|\nabla f(x)\|
\le (2/C) \|\nabla f(x)\|.
\end{align}

By \lemref{federer4.8(2)}, $x_\eta -x$ is parallel to $\nabla f(x_\eta)$, so we can write 
\begin{align}\label{Peta1}
x_\eta -x = c_\eta \nabla f(x_\eta),
\end{align}
for some scalar $c_\eta > 0$. 
Using a Taylor expansion, we have for some $s_\eta\in(0,1)$,
\begin{align}\label{Peta2}
\eta 
&= f(x_\eta) - f(x) \\
&= [x_\eta - x]^\top \nabla f(s_\eta x_\eta + (1-s_\eta) x) \nonumber\\
& = c_\eta [\nabla f(x_\eta)]^\top \nabla f(s_\eta x_\eta + (1-s_\eta) x).
\end{align}
Extracting the expression of $c_\eta$ from (\ref{Peta2}) and plugging it into (\ref{Peta1}), we get
\begin{align}
\frac{x_\eta - x}\eta = \frac{\nabla f(x_\eta)}{[\nabla f(x_\eta)]^\top \nabla f(s_\eta x_\eta + (1-s_\eta) x)}.
\end{align}
We then develop the denominator using \eqref{kappa2}, to get
\begin{align}
& [\nabla f(x_\eta)]^\top \nabla f(s_\eta x_\eta + (1-s_\eta) x) \\
&= [\nabla f(x_\eta)]^\top \big(\nabla f(x_\eta) \pm \kappa_2 \|x_\eta -x\|\big) \\
&= \|\nabla f(x_\eta)\|^2 \pm  \kappa_2 \|x_\eta -x\| \|\nabla f(x_\eta)\|.
\end{align}
Hence,
\begin{align}
\frac{x_\eta - x}\eta 
&= \frac{\nabla f(x_\eta)}{\|\nabla f(x_\eta)\|^2 \pm  \kappa_2 \|x_\eta -x\| \|\nabla f(x_\eta)\|} \\
&= \frac{F(x_\eta)}{1 \pm  \kappa_2 \|x_\eta -x\|/\|\nabla f(x_\eta)\|}.
\end{align}


By \eqref{kappa2} and \eqref{P_proof1},
\begin{align}
\|\nabla f(x_\eta)\|
&\ge \|\nabla f(x)\| - \kappa_2 \|x_\eta -x\| \\
&\ge (1 - \kappa_2 (2/C)) \|\nabla f(x)\| \\
&\ge \tfrac12 \|\nabla f(x)\|,
\end{align}
for $C$ large enough, which then implies
\begin{align}
\|x_\eta -x\|/\|\nabla f(x_\eta)\|
\le 4/C,
\end{align}
so that, taking $C$ as large as needed, we have
\begin{align}
\label{P_proof2}
\frac{x_\eta - x}\eta = F(x_\eta) \left(1 \pm 2 \kappa_2 \|x_\eta -x\|/\|\nabla f(x_\eta)\|\right).
\end{align}

Continuing,
\begin{align}
F(x_\eta) 
&= F(x) + \int_0^1 D F(s x_\eta + (1-s) x) (x_\eta-x) \d s \\
&= F(x) \pm C_1 \|x_\eta-x\|/\|\nabla f(x)\|^2,
\end{align}
based on \lemref{F}. 
Plugging this into \eqref{P_proof2} above, and also using the fact $\|F(x)\| = 1/\|\nabla f(x)\|$, gives
\begin{align}
x_\eta 
&=  x + \eta F(x) \pm C_2 \eta \|x_\eta - x\|/\|\nabla f(x)\|^2 \\
&=  x + \eta F(x) \pm 2 C_2 \eta^2/\|\nabla f(x)\|^3,
\end{align}
using the first inequality in \eqref{P_proof1}.
\end{proof}

\begin{lemma}
\label{lem:cluster modes}
Take any level $0 < t < \max f$, and any cluster at level $t$, meaning, any connected component of $\up_t$. Then that cluster contains at least one mode. Moreover, take any gradient ascent line that ends at a mode: if it intersects that cluster, then it must remain in that cluster and, therefore, end at one of the modes the cluster contains.
\end{lemma}

\begin{proof}
Take any connected component of $\up_t$ and denote it by $\cC_t$. 

Since $\cC_t$ is a compact set, $\cM := \argmax_{x\in\cC_t}f(x)$ is a non-empty subset of $\cC_t$. Moreover, we claim that any point in $\cM$ is a mode. This is obvious if $\cC_t$ is a singleton, therefore, assume this is not the case. Then $\cC_t^\circ$ is a connected component of $\up_t^\circ = \{f > t\}$. And a point in $\cM$ must be in that interior, by definition, since any point on the boundary of $\cC_t$ is at level $t$. Therefore, $\cM$ is the set of maximizers in $\cC_t^\circ$ and from that we deduce that any point in $\cM$ is a local maximum, i.e., a mode.
Hence, we have established that $\cC_t$ contains at least one mode. 

Take any gradient ascent line, say $\gamma_x$ as defined in \eqref{gradient_flow}. Suppose that this line intersects $\cC_t$, so that there is $s \ge 0$ such that $\gamma_x(s) \in \cC_t$. Then for $r \ge s$, $f(\gamma_x(r)) \ge f(\gamma_x(s)) = t$, and therefore $\{\gamma_x(r): r \ge s\} \subset \up_t$. And since this is a continuous piece of curve, it must be entirely included in a connected component of $\up_t$, and that component must be $\cC_t$ because of the existing intersection at $\gamma_x(s)$.
Finally, by assumption, $\gamma_x$ converges to a mode, so that if it remains in $\cC_t$ past a certain point, the mode it converges to must belong to $\cC_t$.
\end{proof}

\begin{lemma}
\label{lem:converge mode}
Let $x_*$ be a mode of $f$ and let $x$ be in the basin of $x_*$ but $x\neq x_*$. Then
\begin{equation}
\zeta_x(t) \longrightarrow x_* \quad \text{as} \quad  t\nearrow\,f(x_*)-f(x).
\end{equation}
\end{lemma}
\begin{proof}
Let $\gamma_x$ be the integral curve induced by the vector field $\nabla f$ originating from $x$, as defined in \eqref{gradient_flow}.
Notice that for any $\tau\in[0,\infty)$,
\begin{align}
\label{tau_def}
f(\gamma_x(\tau)) - f(x) = \int_0^\tau \nabla f(\gamma_x(s) )^\top \dot \gamma_x(s) \d s = \int_0^\tau \| \nabla f(\gamma_x(s) ) \|^2 \d s =: t(\tau).
\end{align}
Since $\| \nabla f(\gamma_x(s) ) \|>0$ for any $s\in[0,\infty)$, $t$ is a strictly increasing function of $\tau$, and it has an inverse, denoted by $\tau(t)$, satisfying $\tau(0)=0$. The definition in \eqref{tau_def} and relation between $t$ and $\tau$ give that 
\begin{equation}
\lim_{\tau\to\infty} t(\tau) = f(x_*) - f(x) \quad \text{and} \quad \lim_{t\nearrow [f(x_*)-f(x)]} \tau(t) = \infty.
\end{equation} 
Notice that $t$ is differentiable as a function of $\tau$, and $\dot t(\tau) = \| \nabla f(\gamma_x(\tau) ) \|^2$. Hence $\tau$ is also differentiable as a function of $t$, and
\begin{equation}
\dot \tau(t) = \| \nabla f(\gamma_x(\tau(t)) ) \|^{-2}.
\end{equation}
This then leads to
\begin{align}
\frac{\partial \gamma_x(\tau(t))}{\partial t} = \dot \gamma_x(\tau(t)) \, \dot \tau(t) = \frac{\nabla f(\gamma_x(\tau(t)))}{\|\nabla f(\gamma_x(\tau(t)))\|^2},
\quad \text{with} \quad
\gamma_x(\tau(0)) = \gamma_x(0) = x,
\end{align}
which implies that $\zeta_x(t) = \gamma_x(\tau(t))$ for all $t\in[0,f(x_*)-f(x))$. Therefore
\begin{equation*}
\lim_{t\nearrow\, f(x_*)-f(x)} \zeta_x(t) = \lim_{t\nearrow\, f(x_*)-f(x)} \gamma_x(\tau(t))= \lim_{\tau\to\infty} \gamma_x(\tau)= x_*.
\qedhere
\end{equation*}
\end{proof}

\begin{lemma}
\label{lem:cluster balls}
Let $x_*$ be a mode of $f$ and let $t_* = f(x_*)$. For $0 < t < t_*$, let $\cC_t$ denote the connected component of $\up_t$ that contains $x_*$. Then there are non-increasing functions $\delta_-$ and $\delta_+$ defined on $(0, t_*)$ such that $0 < \delta_- < \delta_+$ and $\delta_+(t) \to 0$ as $t \to t_*$, and 
\begin{align}
\label{Ct_bounds}
\bar\ball(x_*, \delta_-(t)) \subset \cC_t \subset \ball(x_*, \delta_+(t)), \quad \text{for all } 0 < t < t_*.
\end{align}
In particular, one may take $\delta_-$ and $\delta_+$ such that $\delta_-(t) \asymp \delta_+(t) \asymp \sqrt{t_* - t}$ as $t \nearrow t_*$.
\end{lemma}

\begin{proof}
We first focus on \eqref{Ct_bounds}. Notice that this part of proof only needs the modes to be isolated, but not necessarily non-degenerate, to accompany our comments in \secref{extensions}. For a fixed constant $a>1$, define
\begin{align}
\delta_-(t) = \sup\{\delta\ge0: \ball(x_*,\delta)\subset \cC_t\}\quad \text{and} \quad \delta_+(t) = a\sup_{x\in \cC_t} \|x-x_*\|.
\end{align} 
Due to the fact that $\cC_t \subset \cC_{t^\prime}$ for any $0<t^\prime \le t<t_*$, $\delta_-$ and $\delta_+$ are non-increasing on $(0,t_*)$.  It is also clear that $\delta_- < \delta_+$ and \eqref{Ct_bounds} is satisfied, using the compactness of $\cC_{t}$. 

Next we show that $\delta_-(t)>0$ for all $t\in(0,t_*)$. Let $\cB_t = \cL_t \cap \cC_t$. Denote $\tilde \up_t := \{x : f(x) > t\}$. Then note that $\up_t = \tilde \up_t \cup \cL_t$, and $\cC_t = (\cC_t\cap \tilde \up_t)\cup \cB_t$. It is also clear that $\tilde \up_t \cap \cC_t \subset \cC_t^\circ$. Since $\cC_t$ is a compact set, we have $\partial \cC_t = \cC_t \setminus \cC_t^\circ \subset \cB_t$, that is, $f(x)=t$ for all $x\in\partial \cC_t$. Define $\tilde\delta_-(t) : = \inf_{x\in \cB_t} \|x-x_*\|$. For any $t\in(0,t_*)$, suppose that there exists $x\in \bar\ball(x_*,\tilde\delta_-(t))$ such that $f(x)<t$. By the continuity of $f$, there exists an interior point $\tilde x$ of the line segment $[x,x_*]$ such that $\tilde x\in \partial\cC_t\subset \cB_t.$ This would lead to a contradictory result $\|\tilde x-x_*\| < \tilde\delta_-(t) $. Hence we must have $f(x)\ge t$ for all $x\in \bar\ball(x_*,\tilde\delta_-(t))$, which yields $\delta_-(t) \ge \tilde \delta_-(t)$ by definition. If $ \tilde \delta_-(t)=0$, then $x_*\in \cB_t$ implying that $f(x_*) = t$, which is not true. Therefore we have $0<\tilde \delta_-(t) \le \delta_-(t) .$ 

Then we show that $\delta_+(t) \to 0$ as $t \to t_*$. Since $x_*$ is an isolated mode of $f$, there exists $\delta_0>0$ such that $x_*$ is the only mode in $\bar\ball(x_*,\delta_0)$. For any $\delta\in(0,\delta_0)$, let $t_\diamond(\delta)=\sup_{x\in\partial \ball(x_*,\delta)} f(x)$ and $t_\dagger(\delta) = \frac{1}{2}(t_\diamond(\delta) + t_*)$. Note that $t_\diamond(\delta)<t_\dagger(\delta) < t_*$, and $\cC_{t_\dagger(\delta)}\subset \bar\ball(x_*,\delta)$, with the latter being a result of the following argument: if there exists $x\in\cC_{t_\dagger(\delta)}$ such that $\|x-x_*\|>\delta$, then there exists a path $p[x,x_*]$ connecting $x$ and $x_*$ such that $p[x,x_*]\subset \cC_{t_\dagger(\delta)}$; however, this path $p[x,x_*]$ must intersect $\partial \ball(x_*,\delta)$, resulting in $\sup_{x\in\partial \ball(x_*,\delta)} f(x) \geq t_\dagger(\delta) > t_\diamond(\delta)$, which is not consistent with the definition of $t_\diamond$. Therefore 
\begin{equation}
\label{deltaplug_bound}
\delta_+(t_* ) \le \delta_+(t_\dagger(\delta)) \le a\delta.
\end{equation}
Since $\delta_+$ is non-increasing and non-negative, there exists $\delta_*\ge 0$ such that  $\delta_+(t) \to \delta_*$ as $t \to t_*$. By observing that $\delta$ in \eqref{deltaplug_bound} can be chosen arbitrarily small, we conclude that $\delta_*=0$.

Under the assumption that $x_*$ is non-degenerate and the second derivatives of $f$ are continuous, there exist $\lambda_1\ge \lambda_0>0$ and $\delta_1>0$ such that for all $x\in\bar\ball(x_*,\delta_1)$, all the eigenvalues of $Hf(x)$ are within $[-\lambda_1,-\lambda_0]$. Using a Taylor expansion, we can write for all $x\in\bar\ball(x_*,\delta_1)$,
\begin{equation}
-\frac{1}{2}\lambda_1\|x-x_*\|^2 \le f(x) - f(x_*) \leq -\frac{1}{2}\lambda_0\|x-x_*\|^2.
\end{equation}
It then follows that for any $t\in(t_* - \frac{1}{2}\delta_1\lambda_0,t_*)$, 
\begin{equation}
\label{Ct_inclusion}
\bar\ball\Big(x_*, \sqrt{2\lambda_1^{-1}(t_*-t)}\,\Big) \subset \cC_t \subset \bar\ball\Big(x_*, \sqrt{2\lambda_0^{-1}(t_*-t)}\,\Big),
\end{equation}
which gives the desired result in \eqref{Ct_bounds}.
\end{proof}

%

\subsection{Proof of \thmref{alg1}}
\label{sec:proof1}

Let $\zeta(t) := \zeta_x(t-t_0)$, where $\zeta_x$ is the flow defined in \eqref{gradient_flow_norm}, that is, the gradient ascent line originating with $x$ parameterized by the level. In what follows, we only consider $t$ in $[t_0, t_*]$, where $t_* := f(x_*)$, and note that $f(\zeta(t)) = t$ for any such $t$. For $t$ in that range, we let $\cZ_t := \zeta([t_0,t])$,  which is the gradient line as a subset of the ambient Euclidean space starting at $x$ and ending when reaching level $t$.

For $k$ such that $t_k < t_*$, define $z_k := \zeta(t_k)$, and note that $f(z_k) = t_k$, so that $z_k \in \level_{t_k}$, just like $q_k$.
Of course, as the level grid becomes finer and finer, the sequence $(z_k)$ is closer and closer to the gradient ascent line, and the basic idea is to compare the sequence $(q_k)$ to the sequence $(z_k)$. Note that $z_0 = x = q_0$ --- which is a good start.

\subsubsection{Algorithm returns $x_*$}
\label{sec:proof1 return}
It suffices to show that the sequence reaches a leaf cluster. After that, it has to remain there by construction, which forces it to end at a leaf cluster, since all descendants of a leaf cluster are also leaf clusters, by definition. 
Take $\s \in (0, t_*)$ close enough to $t_*$ that $\cC_* := \cC_\s(x_*)$ is a leaf cluster. That such a level $\s$ exists comes from \lemref{cluster balls} and the fact that the modes are isolated.

Suppose $\eta$ is small enough that $\eta \le \frac12 (t_* - \s)$, and let $t_\# := \frac12 (t_*+\s)$.
Define
$
\nu := \frac12 \min\{\|\nabla f(z)\| : z \in \cZ_{t_\#}\},
$
and note that $\nu > 0$ by the fact that $t \mapsto \|\nabla f(\zeta(t))\|$ is continuous and positive on $[t_0,t_\#]$ because the gradient line traced by $\zeta$ does not contain a critical point other than $x_*$ at its very end (at $t = t_* > t_\#$). 
By an application of \eqref{kappa2}, we have that $\|\nabla f(y)\| \ge \nu$ for all $y$ in the `tube' $\cT := \ball(\cZ_{t_\#}, \delta_{\rm tube})$, where $\delta_{\rm tube} :=  \nu/\kappa_2$.
Let $k_\# := \max\{k : t_k \le t_\#\}$ and note that $t_\# -\eta < t_{k_\#} \le t_\#$ and $(t_\# - t_0)/\eta -1 < k_\# \le (t_\# - t_0)/\eta$. 

We show below that, when $\eta$ is small enough, the sequence $(q_k)$ is uniquely defined up to $k = k_\#$ at which step it is inside $\cC_*$, which is certainly enough to establish that the algorithm returns $x_*$.

\paragraph{The sequence $(z_k)$}
We first note that $\zeta$ is twice differentiable, with $\dot\zeta(t) = F(\zeta(t))$ and $\ddot\zeta(t) = D F(\zeta(t)) \dot\zeta(t) =  D F(\zeta(t)) F(\zeta(t))$. In view of \lemref{F} and the fact that $\|\nabla f(\zeta(t))\| \ge 2\nu$ for $t \in [t_0,t_\#]$, for such a $t$ it holds that $\|\ddot\zeta(t)\| \le (3 \kappa_2/(2 \nu)^2) (1/2\nu)$.
Therefore, a Taylor expansion gives 
\begin{align}
z_k - z_{k-1}
&= \zeta(t_k) - \zeta(t_{k-1}) \\
&= (t_k - t_{k-1}) \dot\zeta(t_{k-1}) \pm C_1 (t_k-t_{k-1})^2 \\
&= \eta F(z_{k-1}) \pm C_1 \eta^2, \label{z diff}
\end{align}
for any $k \le k_\#$.

In addition, the sequence $(z_k)$ does not come `close' to any cluster at $\s$ other than $\cC_*$. Indeed, $\cZ_{t_*}$ does not intersect any of these clusters, for otherwise it would end at one of the corresponding modes and not at $x_*$ --- see \lemref{cluster modes}. 
Therefore, by compactness of the leaf clusters and of $\cZ_{t_*}$, there is $\delta_* > 0$ such that $\cZ_{t_*}$ is at least $\delta_*$ away from any cluster at $\s$ other than $\cC_*$, and this then applies to the sequence $(z_k)$ since $(z_k) \subset \cZ_{t_*}$.

\paragraph{The sequence $(q_k)$}
Suppose for now that, for $k \le k_\#$, $d_{k-1} := \|q_{k-1} - z_{k-1}\| \le \delta_{\rm tube}$ --- so that $\|\nabla f(q_{k-1})\| \ge \nu$ since $z_{k-1} = \zeta(t_{k-1})$ with $t_{k-1} \le t_\#$ --- and consider bounding $d_k = \|q_k - z_k\|$.
Assume $\eta$ is small enough that 
\begin{align}
\label{eta_bound}
\eta \le \nu^2/C_2, \text{ where $C_2$ is the constant of \lemref{P}.}
\end{align} 
The same lemma then tells us that $q_k$ is well-defined and gives
\begin{align}
q_k 
&= q_{k-1} + (t_k - t_{k-1}) F(q_{k-1}) \pm (t_k - t_{k-1})^2 C_2/\|\nabla f(q_{k-1})\|^3 \\
&= q_{k-1} + \eta F(q_{k-1}) \pm C_3 \eta^2. \label{q diff}
\end{align}
The bound \eqref{q diff} combined with \eqref{z diff} gives
\begin{align}
q_k - z_k
&= q_k - q_{k-1} + q_{k-1} - z_{k-1} + z_{k-1} - z_k \\
&= \eta F(q_{k-1}) \pm C_3 \eta^2 + q_{k-1} - z_{k-1} - \eta F(z_{k-1}) \pm C_1 \eta^2.
\end{align}
Applying the triangle inequality, we further derive
\begin{align}
d_k
&\le d_{k-1} + \eta \|F(q_{k-1}) - F(z_{k-1})\| + (C_1+C_3) \eta^2 \\
&\le d_{k-1} + \eta (3\kappa_2/\nu^2) \|q_{k-1} - z_{k-1}\| + C_4 \eta^2 \\
&= (1+ C_5 \eta) d_{k-1} + C_4 \eta^2, \label{d_bound}
\end{align}
where the second inequality is due to the fact that $\|D F(y)\| \le 3\kappa_2/\nu^2$ inside $\cT$ and the segment $[z_{k-1},q_{k-1}]$ is in that set since that set contains $\ball(z_{k-1}, \delta_{\rm tube})$ and $\|z_{k-1} - q_{k-1}\| = d_{k-1} \le \delta_{\rm tube}$, by assumption.

Define the sequence 
\begin{align}
\text{$a_0 = 0$ and $a_k = (1+C_5 \eta) a_{k-1} + C_4 \eta^2$ for $k \ge 1$.} 
\end{align}
As is well-known, 
\begin{align}
\label{a_bound}
a_k 
\le C_4 \eta^2 \frac{\exp[C_5 \eta k] - 1}{C_5 \eta}, \quad \forall k \ge 0.
\end{align}
If we only consider $0 \le k \le k_\#$, it holds that $a_k \le C_6 \eta$, since over that range $\eta k \le \eta k_\# \le t_\# - t_0 \le t_*$. 

We choose $\eta$ small enough as to guarantee that $C_6 \eta \le \delta_{\rm tube}$. This is in addition to the main bound on $\eta$ assumed in \eqref{eta_bound}.
For a recursion, we start at $d_0 = \|q_0 - z_0\| = 0$ and assume that $d_{k-1} \le a_{k-1}$. Then \eqref{d_bound} gives 
\begin{align}
d_k
&\le (1+ C_5 \eta) d_{k-1} + C_4 \eta^2 \\
&\le (1+ C_5 \eta) a_{k-1} + C_4 \eta^2 \\
&= a_k,
\end{align}
and thus the recursion can proceed all the way to $k = k_\#$, thus establishing that
\begin{align}
\label{d_final}
d_k = \|q_k - z_k\| \le C_6 \eta, \quad \forall k = 0, \dots, k_\#.
\end{align}

\paragraph{Conclusion}
We have thus shown that, if $\eta$ is small enough, the sequence $(q_k : k = 0, \dots, k_\#)$ is well-defined and satisfies  \eqref{d_final}.
In particular, because $\|q_{k_\#} - z_{k_\#}\| \le C_6 \eta$, by \eqref{kappa1} we have 
\begin{align}
\label{alg1_proof1}
f(q_{k_\#}) 
\ge f(z_{k_\#}) - C_7 \eta
= t_{k_\#} - C_7 \eta
\ge t_\# - C_8 \eta.
\end{align}
When $\eta$ is so small that $t_\# - C_8 \eta \ge \s$, we thus have that $q_{k_\#} \in \up_{\s}$. 
And when $\eta$ is so small that $\delta_* - C_6 \eta > 0$, the connected component of $\up_{\s}$ that $q_{k_\#}$ belongs to must be $\cC_*$, since $z_{k_\#}$ is at least $\delta_*$ away from any other component of $\up_{\s}$.
Thus, the sequence $(q_k)$ enters $\cC_*$.

\subsubsection{Convergence}
\label{sec:proof1 convergence}
Given $\delta > 0$, we show that when $\eta$ is small enough, the Hausdorff distance between the polygonal line $\cQ := \bigcup_k [q_{k-1}, q_k]$ and $\cZ_{t_*}$ is bounded from above by $\delta$.
 
We showed that, when $\s$ is below but close enough to $t_*$, and when $\eta$ is sufficiently small, we have $q_k \in \cC_* = \cC_\s(x_*)$ for all $k \ge k_\#$. Hence, if $\delta_+$ is as in \lemref{cluster balls}, then any such $q_k$ belongs to $\ball(x_*, \delta_+(\s))$. By convexity of the ball, we must have that past $q_{k_\#}$, $\cQ$ is inside the same ball.
We also saw that $z_{k_\#} \in \cC_*$, and therefore past that point, $\cZ_{t_*}$ must be inside $\cC_\s(x_*)$ (since it's the gradient ascent line we are dealing with), and thus inside that same ball.
Therefore, past step $k_\#$, both the polygonal line and the gradient line are inside a ball of radius $\delta_+(\s)$, and must therefore be within Hausdorff distance $2\delta_+(\s)$. 
We now choose $\s$ close enough to $t_*$ that $2\delta_+(\s) \le \delta$.

Now, take any $k = 1, \dots, k_\#$ and any point $q \in [q_{k-1}, q_k]$.
We have
\begin{align}
\|q-z_{k-1}\|
&\le \|q-q_{k-1}\| + \|q_{k-1} - z_{k-1}\| \\
&\le \|q_k-q_{k-1}\| + d_{k-1},
\end{align}
with, based on \eqref{q diff}, 
\begin{align}
\label{q diff convergence}
\|q_k-q_{k-1}\|
\le \eta \|F(q_{k-1})\| + C_3 \eta^2
\le \eta/\nu + C_3 \eta^2
\le C_9 \eta,
\end{align}
and with \eqref{d_final} giving $d_{k-1} \le C_6 \eta$.
Hence, $\|q-z_{k-1}\| \le C_{10} \eta$. This is true for any point $q$ on $\cQ$ up to step $k_\#$. And because $(z_k) \subset \cZ_{t_*}$, this implies that this part of $\cQ$ is within $C_{10} \eta$ of $\cZ_{t_*}$.

And the same is true the other way around. To be sure, we walk the reader through the same arguments. 
Take any $k = 1, \dots, k_\#$ and any point $z = \zeta(t)$ with $t \in [t_{k-1}, t_k]$.
We have
\begin{align}
\|z-q_{k-1}\|
&\le \|z-z_{k-1}\| + \|z_{k-1} - q_{k-1}\|,
\end{align}
with, exactly as in \eqref{z diff}, 
\begin{align}
\label{z diff convergence}
\|z-z_{k-1}\|
&\le (t-t_{k-1}) \|F(z_{k-1})\| + C_1 (t-t_{k-1})^2 \\
&\le \eta \|F(z_{k-1})\| + C_1 \eta^2 
\le \eta/\nu + C_1 \eta^2
\le C'_9 \eta,
\end{align}
and with \eqref{d_final} giving $\|z_{k-1} - q_{k-1}\| = d_{k-1} \le C_6 \eta$, as before.
Hence, $\|z-q_{k-1}\| \le C'_{10} \eta$. This is true for any point $z$ on $\cZ_{t_*}$ up to $z_{k_\#}$. And because $(q_k) \subset \cQ$, this implies that this part of $\cZ_{t_*}$ is within $C'_{10} \eta$ of $\cQ$.

By choosing $\eta$ even smaller if needed that $C_{10} \eta \le \delta$ and $C'_{10} \eta \le \delta$, we get that $\cQ$ and $\cZ_{t_*}$, in their entirety, are within Hausdorff distance $\delta$.

\subsubsection{Bound on the Hausdorff distance}
\label{sec:proof1 Hausdorff bound}
We assume here that $x_*$ is non-degenerate. In that case, by the assumption that the density is twice continuously differentiable, there is $\delta_0$ such that, for some $\lambda_0 > 0$, $H f$ has all its eigenvalues bounded from above by $- \lambda_0$ inside the ball $B_0 := 
\ball(x_*, \delta_0)$. 
Now, consider $g(t) := \|\nabla f(\zeta(t))\|^2$. It is differentiable, with 
\begin{align}
g'(t) 
&= \frac{2 \nabla f(\zeta(t))^\top Hf(\zeta(t)) \nabla f(\zeta(t))}{\|\nabla f(\zeta(t))\|^2}
\end{align}
so that $g'(t) \le -2\lambda_0$ whenever $t \ge s_0 := \sup\{s: \zeta(s) \notin B_0\}$.
From this we deduce that $g(t) \ge 2\lambda_0 (t_*-t)$, i.e., for $t \in [s_0, t_*]$,
\begin{equation}
\label{gr_norm_bound}
\|\nabla f(\zeta(t))\| \ge \sqrt{2\lambda_0 (t_*-t)}.
\end{equation}

Define $\nu_0 := \frac12 \min\{\|\nabla f(\zeta(s))\| : 0 \le s \le s_0\}$. Note that $\nu_0 > 0$.
In what follows, we take $\s = t_* - 2a_* \eta$, where $a_* > 0$ will be chosen large enough below but fixed.
Note that $t_\# = t_* - a_* \eta$.
For $\eta$ small enough that $\cC_{t_\#} \subset \ball(x_*, \delta_0)$, we also have 
\begin{align}
2 \nu 
&= \min\{\|\nabla f(\zeta(s))\| : 0 \le s \le t_\#\} \\
&= \min\{\|\nabla f(\zeta(s))\| : 0 \le s \le s_0\} \wedge \min\{\|\nabla f(\zeta(s))\| : s_0 \le s \le t_\#\} \\
&\ge 2 \nu_0 \wedge \sqrt{2\lambda_0 (t_*-t_\#)}.
\end{align}
Therefore, for $\eta$ small enough, $\nu \ge  \sqrt{\lambda_0 a_* \eta}$.

Now, tracing back the necessary conditions on $\eta$ for the arguments in \secref{proof1 return} to work, we see that the more stringent requirement, at least when $\nu$ is small enough, is \eqref{eta_bound}. Because of the lower bound on $\nu$ above, this requirement is satisfied when $1 \le \lambda_0 a_*$, in effect, when $a_*$ is large enough. 

Taking $a_*$ as such, we now follow the arguments in \secref{proof1 convergence}. Note that, here, $\delta_+(t) \asymp \sqrt{t_*-t}$, so that $2 \delta_+(s_*) \asymp \sqrt{\eta}$ for the choice of $s_*$ we made above.  
Up to step $k_\#$, we have that $\cQ$ is within $C_{10} \eta$ from $\cZ_{t*}$. Therefore, $\cQ$ is within distance $O(\sqrt{\eta} + \eta) = O(\sqrt{\eta})$ of $\cZ_{t_*}$.
And vice-versa, $\cZ_{t_*}$ is within Hausdorff distance $O(\sqrt{\eta})$ of $\cQ$.

\subsubsection{Uniform convergence}
\label{sec:proof1 uniform}

We continue to use the same notation as in this whole section, except that we make the dependence on the starting point $x$ explicit whenever needed as in, e.g., $\nu(x)$ denoting $\nu$ when associated with $x$. Let $\cA_*$ be the basin of attraction associated with $x_*$, that is, $\cA_* = \{x\in\bbR^d: \lim_{t\to\infty} \gamma_x(t) = x_*\}$. Note that $\cA_*$ is an open set. We fix $s_*\in(0,t_*)$ close enough to $t_*$ that $\cC_*$ is a leaf cluster. Denote $\cC_\# := \cC_{t_\#}(x_*)$.

First we show that $\nu$ is continuous on $\cA_*\setminus \cC_\#$. Let $x$ be any point in $\cA_*\setminus \cC_\#$. Recall that $\|\nabla f(y)\|\ge \nu(x)$ for all $y\in\cT(x) :=\ball(\cZ_{t_\#}(x), \delta_{\rm tube})$. Notice that $\|F(y)\|\le 1/\nu(x)$ and $\|D F(y)\| \le 3\kappa_2/\nu(x)^2$ for all $y\in\cT(x)$, using \lemref{F}. Define $\cV(x):=\ball(\cZ_{t_\#}(x), \delta_{\rm tube}/2)$. Let $y_1,\dots,y_m$ be a $\delta_{\rm tube}/2$-packing of $\cV(x)$, meaning that $\min_{i \ne j} \|y_i - y_j\| > \delta_{\rm tube}/2$ and also that $\max_{y \in \cV(x)} \min_i \|y - y_i\| \le \delta_{\rm tube}/2$. And define $\cW(x) := \bigcup_i \ball(y_i, \delta_{\rm tube}/2)$. 
Note that $\cW(x)$ is open with $\cV(x) \subset \cW(x) \subset \cT(x)$, so that $F(y)\le 1/\nu(x)$ and $\|D F(y)\| \le 3\kappa_2/\nu(x)^2$ for all $y \in \cW(x)$. If $y,z \in \cW(x)$ are such that $\|y-z\| \le \delta_{\rm tube}/4$, there must be $i$ such that $y,z \in \ball(y_i, \delta_{\rm tube}/2)$, and because that ball is convex, we have
\begin{equation}
\|F(y) - F(z)\| 
\le (3\kappa_2/\nu(x)^2) \|y-z\|.
\end{equation}
If, on the other hand, $\|y-z\| > \delta_{\rm tube}/4$, then we can simply write
\begin{equation}
\|F(y) - F(z)\| 
\le 2/\|\nu(x)\|
= \frac{2/\nu(x)}{\delta_{\rm tube}/4} \delta_{\rm tube}/4
\le \frac{8}{\nu(x) \delta_{\rm tube}} \|y-z\|.
\end{equation}
Hence, $F$ is Lipschitz on $\cW(x)$ with corresponding constant $C := \max\{3\kappa_2/\nu(x)^2, 8/(\nu(x)\delta_{\rm tube})\}$.

Take a positive constant $\delta_\diamond \le \frac{1}{2}\exp\{-C (t_\# - t_0)\}\delta_{\rm tube}.$ 
Suppose that
\[
\cH(y) := \big\{\zeta_y(\tau): \tau \in [0,t_\# - t_0], y \in \ball(x,\delta_\diamond)\big\}  \nsubset \cV(x).
\] 
Then there exist $y\in\ball(x,\delta_\diamond)$ and an escaping time $t_\square\in (0,t_\# - t_0)$ such that $\zeta_y(t_\square)\in\partial \cV(x)$ and $\{\zeta_y(\tau): \tau \in [0,t_\square)\} \subset \cV(x)$. This is impossible because applying a standard result on the dependence of the gradient flow on the initial condition, for example, the main theorem in \citep[Sec 17.3]{hirsch2012differential}, we have
\begin{equation}
\|\zeta_x(\tau_\square) - \zeta_y(\tau_\square)\|
\le \|x-y\| \exp(C \tau_\square) < \frac{1}{2} \delta_{\rm tube},
\end{equation}
which would lead to $\zeta_y(t_\square) \in \cV(x)$, a contradiction against the definition of $t_\square$ as $\cV(x)$ is an open set. 
Therefore we must have $\cH(y) \subset \cV(x)$, and 
\begin{equation}
\label{initial_continuous}
\|\zeta_x(\tau) - \zeta_y(\tau)\|
\le \|x-y\| \exp(C \tau) \le \frac{1}{2} \delta_{\rm tube}, \quad \forall \tau \in [0,t_\# - t_0], \quad \forall y \in \ball(x,\delta_\diamond).
\end{equation}

For $y \in \ball(x,\delta_\diamond)$, without loss of generality, suppose that $t_0 = f(x)\ge f(y)$. Notice that 
\begin{equation}
\cZ_{t_\#}(y) = \cZ_{t_\#-(f(y) - t_0)}(y) \,\bigcup\, \big\{\zeta_{y}(t): \; t\in[t_\# - t_0,t_\# - f(y)]\big\}.
\end{equation}
With notation $R = \sup_{t\in[t_\# - t_0,t_\# - f(y)]} \|\zeta_{y}(t) - \zeta_{y}(t_\# - t_0)\|$, we can write
\begin{align}
\label{haus_dist}
d_H(\cZ_{t_\#}(x), \cZ_{t_\#}(y)) 
&\le  d_H(\cZ_{t_\#}(x), \cZ_{t_\#-[f(y) - t_0]}(y) ) + R \\
&\le  \sup_{t\in[0,t_\# - t_0]}\|\zeta_x(t) - \zeta_{y}(t)\| + R\\
&\le  \|x-y\| \exp\{C (t_\#-t_0)\} + R, \label{haus_dist_last}
\end{align}
where the last equality is a result of \eqref{initial_continuous}.

Below we further require that $\delta_\diamond \le \frac{1}{2\kappa_1}\delta_{\rm tube}\nu(x)$. We will show that 
\begin{equation}
\label{Gy_subset}
\cG_y: = \big\{\zeta_y(t): t\in [t_\# - t_0,t_\# - f(y)]\big\} \subset \cT(x).
\end{equation}
If this is true, then the length of $\cG_y$ is
\begin{align}
{\rm length}(\cG_y)&=  \int_{t_\# - t_0}^{t_\# - f(y)} \|\dot\zeta_y(t)\| \d t \\
&=  \int_{t_\# - t_0}^{t_\# - f(y)} \|\nabla f(\zeta_y(t))\|^{-1} \d t \\
&\le   (f(x) - f(y))/\nu(x) \\
&\le  \kappa_1  \|x-y\|/\nu(x) \\ \label{G_y bound}
&\le  \kappa_1  \delta_\diamond/\nu(x) \le \tfrac{1}{2} \delta_{\rm tube},
\end{align}
where in the 3rd line we used the fact that $t_0 = f(x)$, and in the fourth line we used \eqref{kappa1}.
The above calculation confirms that indeed (\ref{Gy_subset}) must be true, because otherwise there exists an escaping time $t_\triangle \in (t_\# - t_0,t_\# - f(y))$ such that $\zeta_y(t_\triangle) \in \partial \cT$ and $\widetilde \cG_y:=\{\zeta_y(t): t\in [t_\# - t_0,t_\triangle)\} \subset \cT$, which will lead to a contradiction with the fact that $\zeta_y(t_\# - t_0) \in \cV(x)$ and ${\rm length}(\widetilde \cG_y) < \frac{1}{2}\delta_{\rm tube}$, with the latter following from a similar calculation as above after replacing $t_\# - f(y)$ by $t_\triangle$.
Now, using \eqref{G_y bound} along the way,
\begin{align}
R 
%
 %
&= \sup_{t\in[t_\# - t_0,t_\# - f(y)]} \Big\|\int_{t_\# - t_0}^t \dot\zeta_y(t) \d t \Big\| \\
&\le {\rm length}(\cG_y) \\
&\le \frac{\kappa_1}{\nu(x)}\|x-y\|.
\end{align}

Combining this with \eqref{haus_dist_last}, we obtain
\begin{equation}
d_H(\cZ_{t_\#}(x), \cZ_{t_\#}(y)) \le C_\dagger\|x-y\|,
\end{equation}
where $C_\dagger = \exp\{C (t_\#-t_0)\} + \kappa_1/\nu(x)$. Hence
\begin{align}
&|\nu(x) - \nu(y)| \\
&\le  \max\Big\{\sup_{v\in \cZ_{t_\#}(x)} \; \inf_{w\in \cZ_{t_\#}(y)} | \|\nabla f(v)\| - \|\nabla f(w)\| |,\; \sup_{v\in \cZ_{t_\#}(y)} \; \inf_{w\in \cZ_{t_\#}(x)} | \|\nabla f(v)\| - \|\nabla f(w)\| | \Big\} \\
&\le  \kappa_2 d_H(\cZ_{t_\#}(x), \cZ_{t_\#}(y)) \\
&\le  \kappa_2C_\dagger\|x-y\|.
\end{align}
We have shown that $\nu$ is a continuous function on $\cA_*\setminus \cC_\#$. 

Let $\cA$ be the union of the basins of attraction for all the local modes, and $\cB = \cA^{\complement}$, which is also the boundary of $\cA$. By \lemref{basins}, $\cB$ has zero Lebesgue measure. Let $\cC$ be the union of all the leaf clusters. For $t,\delta> 0$, define $\Gamma_{\delta,t} = \up_t \bigcap \ball(\cB,\delta)^{\complement} \setminus \cC^\circ$, which is a compact set. Define $\beta(\delta,t): = \inf_{x\in \Gamma_{\delta,t}} \nu(x)$, which is positive for any $t>0$ and $\delta>0$ small enough that $\Gamma_{\delta,t}$ is not empty, since $\Gamma_{\delta,t}$ is a compact set, and $\nu$ is continuous on an open set containing $\Gamma_{\delta,t}$ as shown above. Based on the proof of \thmref{alg1}, in order to guarantee that Algorithm 1 returns the correct mode for any $x\in \Gamma_{\delta,t}$, we only need to choose $\eta>0$ small enough that $\eta\le C_0 \beta(\delta,t)^2$ for some constant $C_0>0$. For an arbitrarily small but fixed $\epsilon>0$, let $t_\epsilon$ be the largest $t>0$ such that the probability measure of $\up_t^{\complement}$ is not larger than $\epsilon$, and let $\delta_\eta$ be the smallest $\delta>0$ such that $\eta\le C_0 \beta(\delta,t_\epsilon)^2$, and define $\Omega_\eta = \Gamma_{\delta_\eta,t_\epsilon}\bigcup \cC^\circ=\up_{t_\epsilon} \bigcap \ball(\cB,\delta)^{\complement}$. Using parameter $\eta$, for any starting point $x\in\Omega_\eta$, we always have the correct clustering result using Algorithm 1. Since $\beta_\epsilon(\delta):=\beta(\delta,t_\epsilon)$ is non-decreasing, $\beta_\epsilon(\delta) \to 0$ as $\delta\to 0$, and $\beta_\epsilon(\delta) >0$ if $\delta>0$, it is clear that $\delta_\eta \to 0$ as $\eta\to0$. Note that any point in $\cB$ is in the basin of attraction of some critical point that is not a local mode. Under the assumption $f(x)\to0$ as $\|x\| \to\infty$, there exists $r<\infty$ large enough  that all the points in $\up_{t_\epsilon}\bigcap\cB$ are in the basins of attraction of the critical points in $\ball(\underline{0},r)$, where $\underline{0}$ is the origin. Using a standard approach we can build a Morse function $\hbar$ that has bounded support and coincides with $f$ on $\ball(\underline{0},r)$. This way we understand $\up_{t_\epsilon}\bigcap\cB$ as a subset of $\cB_{\hbar}$, which denotes the boundary of basins of attraction of all local modes of $\hbar$. Based on \citep[Th 4.2]{banyaga2013lectures}, $\cB_{\hbar}$ consists of finitely many $k$-dimensional manifolds associated with the critical points of $\hbar$ that are not local modes, where $k=0,\cdots,d-1$. Hence the Lebesgue measure of $\up_{t_\epsilon}\bigcap\ball(\cB,\delta_\eta)$ is of order $O(\delta_\eta)=o(1)$ as $\delta_\eta \to 0$. Using the boundedness of $f$, the probability measure of $\up_{t_\epsilon}\bigcap\ball(\cB,\delta_\eta)$ is also of order $o(1)$, so that the probability measure of $\Omega_\eta$ is lower bounded by $1-\epsilon+o(1)$ as $\eta\to0$. Since $\epsilon$ is arbitrarily small, we conclude the proof.

\subsection{Proof of \thmref{alg2}}
\label{sec:proof2}
The proof of \thmref{alg2} is similar to that of \thmref{alg1} given in \secref{proof1}. As much as we can, we reuse the same notation. Of course, here $(q_k)$ denote the sequence generated by \algref{2} instead, and we define $t_k$ as the level of $q_k$ so that $f(q_k) = t_k$. Otherwise, $\zeta$ has the same meaning, and so does $z_k = \zeta(t_k)$, which again has level $t_k$ and is compared with $q_k$ in a recursive manner.

A natural choice for an arbitrary neighborhood is a ball, but to keep the parallel with \secref{proof1}, we use a cluster, and this is justified in view of \lemref{cluster balls}. 
Therefore, take $\s \in (t_0, t_*)$ arbitrarily close to $t_*$, and let $\cC_*$ be as before. We want to show that the sequence enters $\cC_*$, eventually. 
Define $t_\#$, $\nu$, and $\cT$ as before, and note that it remains true that $\|\nabla f(y)\| \ge \nu$ for all $y \in \cT$.
We also define $k_\#$ in exactly the same way, and keep the same notation --- even though things are here `parameterized' by $\eps$ but in view of the definition of $\eta$ below in \eqref{eta def}, this is just fine. Note that we do not yet have a good control over $k_\#$.

Here the levels are not regularly discretized, and so we denote $\eta_k := t_k - t_{k-1}$. By definition and \eqref{kappa1},
\begin{align}
\label{eta def}
0 
\le t_k-t_{k-1}
= f(q_k) - f(q_{k-1})
\le \kappa_1 \|q_k-q_{k-1}\|
\le \kappa_1 \eps
=: \eta.
\end{align}
Therefore, $\eta_k \le \eta$ for all applicable $k$, so that $\eta$ plays the same role as before. We note that $\eta$ is simply proportional to $\eps$, and so saying `when $\eta$ is small enough' is completely equivalent to saying `when $\eps$ is small enough'.
Furthermore, by how the sequence $(q_k)$ is built and by \eqref{taylor2}, if it is the case that $q_{k-1} \in \cT$, then we have
\begin{align}
f(q_k) 
&\ge f(q_{k-1}  + \eps N(q_{k-1})) \\
&\ge f(q_{k-1}) + \eps \|\nabla f(q_{k-1})\| - \tfrac12 \kappa_2 \eps^2 \\
&\ge f(q_{k-1}) + \eps \nu - \tfrac12 \kappa_2 \eps^2 \\
&\ge f(q_{k-1}) + \tfrac12 \eps \nu,
\end{align}
assuming $\eps$ is small enough.
So that we also have the lower bound $\eta_k \ge (\nu/2) \eps$ whenever $q_{k-1} \in \cT$. 

In what follows, not all the numbered constants have the same exact meaning as in \secref{proof1}. We only keep the numbering for the reader's orientation.
Following the proof in \secref{alg1}, instead of \eqref{z diff}, here we have
\begin{align}
z_k - z_{k-1}
&= \eta_k F(z_{k-1}) \pm C_1 \eta^2, \label{z diff2}
\end{align}
valid for any $k \le k_\#$.
And instead of \eqref{q diff}, we get
\begin{align}
q_k - q_{k-1}
&= \eta_k F(q_{k-1}) \pm C_3 \eta^2, \label{q diff2}
\end{align}
We then simply follow the same arguments, augmented by the following. Since in the recursion we assume that $d_{k-1} \le \delta_{\rm tube}$, as part of the recursion we also have $q_{k-1} \in \cT$, as so $\eta_k \ge (\nu/2) \eps$. And since $t_k = \eta_k + \cdots + \eta_1 + t_0$, this implies that $k_\# \le (t_\#-t_0)/(\nu/2) \eps = (t_\#-t_0)/(\nu/2\kappa_1) \eta$.
Thus, the bound $a_{k_\#} \le C_6 \eta$ resulting from \eqref{a_bound} holds in a similar way, because $\eta k_\#$ is upper-bounded by the constant $(t_\#-t_0)/(\nu/2\kappa_1)$. 
We thus arrive at \eqref{d_final}.

The rest of the proof is identical.
We thus conclude that, when $\eps$ is small enough, the sequence $(q_k)$ enters the cluster $\cC_*$. 
In addition, we can show as we did in \secref{alg1} that the polygonal line connecting the sequence $(q_k : k \le k_\#)$ is within distance $O(\eta) = O(\eps)$ of $\cZ_{t_*}$.

It remains to show that the sequence ends at $x_*$ when $\eps$ is small enough; and then to show that the polygonal line converges to the gradient line as $\eps$ approaches 0.

\subsubsection{Algorithm returns $x_*$}
Take $\s > 0$ small enough that $\cC_*$ is a leaf cluster. 
Assume $\eps$ is small enough that $(q_k)$ enters $\cC_*$, and take it even smaller if needed to have $\eps < \delta_*$, where, as before, $\delta_*$ is the minimum separation between two clusters at level $\s$.
Then, because $(f(q_k))$ is non-decreasing and $\|q_k - q_{k-1}\| \le \eps < \delta_*$, once $(q_k)$ enters $\cC_*$, it must remain there.  
It therefore suffices to show that $f(q_k) \to t_*$ since $x_*$ is the only point in $\cC_*$ with a value of $t_*$ and all other points have a smaller value. Let $t_\infty := \sup_k f(q_k) = \lim_k f(q_k)$ and suppose for contradiction that $t_\infty < t_*$. Consider the set $\cA := \{y : \s \le f(y) \le t_\infty\}$. Since $\cA \subset \cC_*$ and $x_* \notin \cA$, $\cA$ has no critical point, and because it is compact, $\nu_\ddag := \inf\{\|\nabla f(y)\| : y \in \cA\} > 0$.
Therefore, if $\eps_\ddag := \min\{\nu_\ddag/\kappa_2, \eps\}$, by \eqref{taylor_N}, for any $y \in \cA$, $f(\tilde y) \ge f(y) + \eta_\ddag$ with $\tilde y := y+ \eps_\ddag N(y)$ and $\eta_\ddag := \nu_\ddag \eps_\ddag - (\kappa_2/2) \eps_\ddag^2 > 0$. 
Now, take $k$ large enough that $f(q_{k-1}) > t_\infty - \eta_\ddag/2$.
We then have $f(\tilde q_{k-1}) \ge f(q_{k-1}) + \eta_\ddag > t_\infty + \eta_\ddag/2$ and $\|\tilde q_{k-1} - q_{k-1}\| = \eps_\ddag \le \eps$. But then the sequence moves to $q_k$ with $f(q_k) \le t_\infty$ when it could have moved instead to $\tilde q_{k-1}$ with $f(\tilde q_{k-1}) > t_\infty$, contradicting the rules governing \algref{2}.

\subsubsection{Convergence}
The convergence of the polygonal line to the gradient line is proved in essentially the same way as in \secref{proof1 convergence}, and we omit further details.

\subsubsection{Bound on the Hausdorff distance}
\label{sec:proof2 Hausdorff bound}
A bound of order $O(\eps)$ on the Hausdorff distance between the polygonal line and the gradient line can be obtained in essentially the same way as in \secref{proof1 Hausdorff bound}, and we omit further details.

\subsubsection{Uniform convergence}
\label{sec:proof2 uniform}

The arguments are truly analogous as those given in \secref{proof1 uniform}, and details are thus omitted.

\subsection*{Acknowledgments}
We are grateful to José Chacón for stimulating discussions.  

\small
\bibliographystyle{chicago}
\bibliography{ref}

\end{document}